\documentclass[12pt, reqno, oneside]{amsart}

\usepackage{enumerate}
\usepackage[nobysame]{amsrefs}
\usepackage{enumitem}
\usepackage{amssymb}
\usepackage{amsmath}
\usepackage{microtype}
\usepackage[right=1.9cm,left=1.9cm, top=1.9cm, bottom=1.8cm,includehead]{geometry}

\theoremstyle{plain}
\newtheorem{theorem}{Theorem}[section]

\newtheorem{lemma}[theorem]{Lemma}
\newtheorem{corollary}[theorem]{Corollary}
\newtheorem{proposition}[theorem]{Proposition}

\theoremstyle{definition}
\newtheorem{definition}[theorem]{Definition}
\newtheorem*{notation}{Notation}

\theoremstyle{remark}
\newtheorem{remark}[theorem]{Remark}
\newtheorem{example}[theorem]{Example}

\newcommand{\ep}{\varepsilon}
\newcommand{\ue}{\mathrm{e}}
\newcommand{\ud}{\mathrm{d}}

\newcommand{\N}{ \in \mathbb{N}}
\newcommand{\Z}{ \in \mathbb{Z}}
\newcommand{\R}{ \in \mathbb{R}}
\newcommand{\eps}{\varepsilon}
\newcommand{\zdef}{{\mathrel{\mathop:}=}}

\newcommand{\norm}[1]{\left\|#1\right\|}

\newcommand{\abs}[1]{\lvert#1\rvert}
\newcommand{\babs}[1]{\bigl|#1\bigr|}
\newcommand{\bgabs}[1]{\biggl|#1\biggr|}
\newcommand{\bnorm}[1]{\bigl\|#1\bigr\|}
\newcommand{\set}[2]{\{#1:#2\}}
\newcommand{\bset}[2]{\bigl\{#1:#2\bigr\}}
\newcommand{\bgset}[2]{\biggl\{#1:#2\biggr\}}
\newcommand{\Bset}[2]{\Bigl\{#1:#2\Bigr\}}
\newcommand{\absp}[1]{\left.\!\!\left\bracevert\!\!\mspace{1mu} \vphantom{Iy}#1 \mspace{1mu}\!\! \right\bracevert\!\!\right.}

\DeclareMathOperator{\FR}{Fr}

\DeclareMathOperator{\supp}{supp}

\title[Almost-periodically driven integrate-and-fire models]{Integrate-and-fire models with\\ an almost periodic input function}

\author{Piotr Kasprzak}
\author{Adam Nawrocki}
\address[A.~Nawrocki and P.~Kasprzak]{Faculty of Mathematics and Computer Science, Adam Mickiewicz University, ul. Umultowska 87, 61-614 Pozna\'n, Poland}

\email[A.~Nawrocki]{adam.nawrocki.amu.edu.pl}
\email[P.~Kasprzak]{kasp@amu.edu.pl}

\author{Justyna Signerska-Rynkowska}
\address[J.~Signerska-Rynkowska]{Faculty of Applied Physics and Mathematics, Gda\'nsk University of Technology, ul. G. Narutowicza 11/12, 80-233 Gda\'nsk, Poland} 

\email[J.~Signerska-Rynkowska]{jsignerska@mif.pg.gda.pl}

\keywords{Displacement map, firing map, firing rate, Haar wavelets, leaky integrate-and-fire model, linear differential equation, mean value, $\mu$-almost periodic function, neuron model, Stepanov almost periodic function, uniformly almost periodic function}
\subjclass[2010]{42A75, 37B55, 37E45, 92B20}

\begin{document}

\begin{abstract}
We investigate leaky integrate-and-fire models (LIF models for short) driven by Stepanov and $\mu$-almost periodic functions. Special attention is paid to the properties of a firing map and its displacement, which give information about the spiking behaviour of the system under consideration. We provide conditions under which such maps are well-defined for every $t \in \mathbb R$ and are uniformly continuous. Moreover, we show that the LIF model with a Stepanov almost periodic input has a uniformly almost periodic displacement map. We also show that in the case of a $\mu$-almost periodic drive it may happen that the displacement map corresponding to the LIF model is uniformly continuous, but is not $\mu$-almost periodic (and thus cannot be Stepanov or uniformly almost periodic). By allowing discontinuous inputs, we generalize some results of previous papers, showing, for example, that the firing rate for the LIF model with a Stepanov almost periodic drive exists and is unique. This is a starting point for the investigation of the dynamics of almost-periodically driven integrate-and-fire systems. The work provides also some contributions to the theory of Stepanov- and $\mu$-almost periodic functions.
\end{abstract}

\maketitle

\section{Introduction}

Integrate-and-fire models are commonly used for modelling the activity of neuronal cells (see for example~\cites{izykiewicz,keener1,tiesinga, ger}).  Although they are not able to capture all the electrophysiological phenomena, as a part of a big neural network, they are computationally more efficient than, for example, the classical Hodgkin-Huxley model (see for example~\cites{hodgin, ger}), and their biological relevance is in some cases satisfactory. In particular, the so-called leaky integrate-and-fire model\footnote{In the special case $\sigma=0$, the LIF model is usually referred to as the \emph{perfect integrator model} (or the PI model for short).}
\begin{equation}\label{lif4}
\dot{x}(t)=-\sigma x(t) +f(t) \qquad \text{for a.e. $t \in \mathbb R$},
\end{equation}
where $\sigma \geq 0$, is one of the models in the center of interest of neuroscientists. This model dates back to Lapicque (see~\cite{lapicque}), who discovered that the voltage $x$ across the cell membrane decays exponentially to its resting state $x_r$ and only an external input $f$, which might be current injected via an electrode or the impulse from a pre-synaptic neuron, might cause the increase of the voltage. Spiking (or firing), that is emitting an action potential, is introduced to this simple dynamics by adding the resetting condition
\begin{equation}\label{lif_spiking}
x(s)=x_{\vartheta} \quad \Longrightarrow \quad \lim_{t\to s^+}x(t)=x_r,
\end{equation}
which says that after the dynamical variable reaches a certain threshold $x_{\vartheta}$ at some time $s$, it is immediately reset to its resting value $x_r$ and the dynamics continues again from the point $(t,x_r)$. Although usually constant thresholds and resets are studied, it is also possible to consider integrate-and-fire models with varying bounds, allowing thus to incorporate into the model additional biological phenomena such as refractory periods and threshold modulation (see~\cite{gedeon}). Such models with time-dependent thresholds and/or resets in some cases can be reduced by an appropriate change of variables to those with constant $x_r$ and $x_{\vartheta}$ (for more details see for example~\cite{brette1}*{Section~2.5}). Therefore, for simplicity, in this paper we will always assume that $x_{\vartheta}=1$ and $x_r=0$.

Since isolated spikes of a given neuron look alike, often it is assumed that the form of the action potential does not carry any information, but rather, it is the structure of the spike train\footnote{A \emph{spike train} is a chain of action potentials emitted by a single neuron.} which matters. Therefore, the idea is to study the properties (and in particular, dynamics) of two special maps associated with the integrate-and-fire models, which carry the information about the distribution of spikings in time, namely, the firing map $\Phi$ and its displacement $\Psi$ (cf.~Definitions~\ref{firingdlalokalnie}~and~\ref{displacement map}). It turns out, for example, that the rotation number of a given point with respect to the mapping $\Phi$ corresponds to the average interspike interval, whereas its multiplicative inverse describes the firing rate (for more details see Section~\ref{sec:firing_rate}).

Despite the fact that integrate-and-fire models are commonly used, their dynamical behaviour has been investigated rigorously in only few papers (see for instance \cites{brette1,gedeon,car-ong}) and usually under the assumption that the forcing term $f$ (or, in the case of general integrate-and-fire models of the form $\dot{x}(t)=F(t,x(t))$, the function $t \mapsto F(t,x)$) is periodic or uniformly almost periodic (see Definition~\ref{Bohr alm per} below). Recently, the broader class of stimulating processes than the periodic ones have been analysed in~\cite{jager2011}. On the other hand, the dynamics of bidimensional adaptive integrate-and-fire models under constant input was studied e.g. in~\cite{TB}. Our aim in this paper is to continue the study of one-dimensional  integrate-and-fire models with (almost) periodic inputs carried out in the articles~\cites{wmjs1, MS, js2015}, and to establish several basic properties of the firing map and its displacement corresponding to the model~\eqref{lif4}--\eqref{lif_spiking} driven by Stepanov or $\mu$-almost periodic functions.

It should be emphasised that in this context it seems to be quite natural to consider almost periodic forcing terms. Indeed, many neurons will respond to the current step with a spike train with (eventually) steady state of periodic firing, and therefore, even if the action potential is generated periodically by each neuron at the pre-synaptic level, the signal a given post-synaptic neuron receives may be no longer periodic as a sum of periodic inputs with incommensurable periods; clearly, such a signal is almost periodic (for more details see~\cite{ger}). One of the ways to qualify neuron's response to such signals, is to measure the so-called firing rate and the regularity of the interspike intervals along the generated post-synaptic signal.

In the paper we work mainly with Stepanov and $\mu$-almost periodic functions, because such functions satisfy quite general regularity conditions (for more details see Section~\ref{sec:2}), and thus we are able to cover quite wide range of LIF models, including models with discontinuous inputs. On the other hand, it seems that without any assumption involving some kind of (almost) periodicity of the forcing term, it would be very hard (if not impossible) to provide a complete description of the behaviour of the firing map and its displacement, since, as we mentioned above, these models are (to some extent) periodic in nature.

The paper is organized as follows. In Section~\ref{sec:2} we recall some basic definitions and facts concerning almost periodic functions. Special attention is paid to Stepanov and $\mu$-almost periodic functions. Section~\ref{sec:3} is devoted to the study of the mean value of $\mu$-almost periodic functions. In particular, it is shown that in general the mean value of $\mu$-almost periodic functions may not exist. In Section~\ref{sec:firing_map} firing map and its displacement are investigated. We begin with presenting some general properties of such maps. For example, we provide the conditions under which $\Phi$ and $\Psi$ are well-defined for every $t \in \mathbb R$ and are (uniformly) continuous (see Proposition~\ref{poprawneokr eslenie} and Theorem~\ref{lem:uniform_con_bound}). Then we move on to the discussion of the firing map and its displacement for LIF models driven by almost periodic functions. Among other things, we show that the LIF model with a Stepanov almost periodic input has a uniformly almost periodic displacement map (see Theorem~\ref{new almost periodic displacement}). We also show that there are LIF models driven by $\mu$-almost periodic functions with (uniformly continuous) displacement maps which fail to be $\mu$-almost periodic (see Example~\ref{ex:mu_no_mu}). In Section~\ref{sec:5} an application of the above-mentioned results to the qualitative theory of almost periodic functions is indicated. Section~\ref{sec:firing_rate} contains a result  on the existence and uniqueness of the firing rate for  LIF models (see Theorem~\ref{thm:LIF_rotation} below), which generalizes analogous result for PI models from~\cite{brette1} and~\cite{wmjs1}. Let us add that, although it seems that the transition from $\sigma=0$ to $\sigma \geq 0$ should be straightforward, it is not and the  proof of Theorem~\ref{thm:LIF_rotation} is based on the result concerning the dynamics of mappings of the real line with almost periodic displacements established by J.~Kwapisz in \cite{kwapisz}. We end the main part of the paper with Section~\ref{sec:7} dealing with the approximation of Stepanov almost periodic functions with Haar wavelets. In the Appendix we return to the mean value of (uniformly) almost periodic functions and present some remarks on its relationship with the antiderivative of such functions. Therefore, apart from our interest in integrate-and-fire models, this work provides a few contributions to the theory of almost periodic functions.

\section{Limit-periodic and almost periodic functions}
\label{sec:2}

In this section we fix notation and recall some basic definitions and facts concerning limit-periodic and almost periodic functions, which will be needed in the sequel.

\begin{notation}
Throughout the paper by $L^0(\mathbb R)$ we will denote the family of all equivalence classes of real-valued Lebesgue measurable functions defined on $\mathbb R$. Furthermore, by $L^p_{\text{loc}}(\mathbb R)$, where $p \in [1,+\infty)$, we will denote the family of all equivalence classes of real-valued functions defined on $\mathbb R$ which are locally Lebesgue integrable with $p$-th power. Very often, by abuse of notation, we will refer to elements of the families $L^0(\mathbb R)$ and $L^p_{\text{loc}}(\mathbb R)$ as functions and we will simply write $f \in L^0(\mathbb R)$ or $f \in L^p_{\text{loc}}(\mathbb R)$. The Lebesgue measure on $\mathbb R$ will be denoted by $\mu$. Given a function $f\colon \mathbb{R}\to \mathbb{R} $ by $f^{\tau}$, where $\tau \R$, we will denote the function $f^{\tau}\colon \mathbb{R}\to \mathbb{R}$ defined by the formula $f^{\tau}(t)=f(t+\tau)$ for $t\R$. 
\end{notation}

At the beginning of this section let us recall the notion of a limit-periodic function.

\begin{definition}\label{def limit periodic}
A function $f\colon \mathbb{R}\to\mathbb{R}$ is called \emph{limit-periodic} (in the sense of Bohr) if it is the limit of a uniformly convergent sequence $(f_n)_{n \in \mathbb N}$ of continuous periodic functions\footnote{Of course, by definition, we assume that all the functions $f_n$ are defined on $\mathbb R$.}.
\end{definition}

\begin{remark}
Clearly, every continuous periodic function is limit-periodic. On the other hand, the function $f \colon \mathbb R \to \mathbb R$ defined by the formula
\[
  f(t)=\sum_{k=1}^{\infty}\frac{1}{k^2}\cos\Bigl(\frac{2\pi t}{2^k}\Bigr) \qquad \text{for $t \in \mathbb R$},
\]
is an example of a limit-periodic function which is not periodic (cf.~\cite{ref}).
\end{remark}

Although there are various classes of almost periodic functions known in the literature (see for example~\cites{andres, BD, BG} and the references therein), in this paper we are going to deal with only three of them; namely, we are going to consider functions almost periodic in the sense of Bohr and Stepanov as well as functions almost periodic in the Lebesgue measure. Let us begin with the definition of a relatively dense set.

\begin{definition}
A set $A \subseteq \mathbb R$ is said to be \emph{relatively dense} if there exists a positive number $l$ such that the intersection $A\cap (x,x+l)$ is non-empty for every $x \in \mathbb R$. In this case, any number $l$ with this property is said to characterise the relative density of the set $A$.
\end{definition}

\begin{definition}\label{Bohr alm per}
A continuous function $f \colon \mathbb{R}\to\mathbb{R}$ is called \emph{almost periodic in the sense of Bohr} (or \emph{uniformly almost periodic}) if for any $\ep>0$ the set of all $\ep$-almost periods of $f$, defined as
\[
 E\{\ep,f\} \zdef \Bset{\tau \in \mathbb R}{\sup_{t \in \mathbb R}\abs{f^{\tau}(t)-f(t)} < \ep}, 
\]
is relatively dense.
\end{definition}

\begin{remark}
Equivalently, Bohr almost periodic functions can be defined as uniform limits of sequences of generalized trigonometric polynomials $P_n(t)=\sum_{j=1}^{k(n)}\bigl(a_j\sin( \lambda_j t)+b_j\cos( \lambda_j t)\bigr)$, where $a_j, b_j\in\mathbb{R}$ and $\lambda_j\in\mathbb{R}$ (see~\cite{corduneanu}). Let us also add that Bohr almost periodic functions are uniformly continuous and bounded (see~\cites{bezykowicz,corduneanu}).
\end{remark}

\begin{remark}
The vector space $AP(\mathbb R)$ of all uniformly almost periodic functions is a Banach space when endowed with the supremum norm $\norm{\cdot}_{\infty}$ (see for example~\cites{corduneanu, Fink,andres}).
\end{remark}

\begin{remark}
Any limit-periodic function is uniformly almost periodic. On the other hand, the class of all limit periodic functions is identical with the class of all uniformly almost periodic functions all whose Fourier exponents are rational multiples of the same number (see~\cite{bezykowicz}*{Theorem, p.~34}). Thus, for instance, the function $f\colon \mathbb R \to \mathbb R$ defined by the formula $f(t)=\cos(\pi t)+\cos(\sqrt{5}t)$, $t \in \mathbb R$, is an example of a Bohr almost periodic function which is not limit-periodic. 
\end{remark}

Now, let us pass to the definition of an almost periodic function in the sense of Stepanov.

\begin{notation}
Given $r>0$ and $p \in [1,+\infty)$, for a function $f \in L^p_{\text{loc}}(\mathbb R)$ let
\[
 \norm{f}_{S^p_r} \zdef  \sup_{t\in\mathbb{R}}\left(\frac{1}{r}\int_{t}^{t+r}\abs{f(u)}^p\, \textup du\right)^{1/p}.
\]
Let us also observe that for every $r_1,r_2>0$ there exist $a,b >0$ such that $a\norm{f}_{S^p_{r_1}} \leq \norm{f}_{S^p_{r_2}}\leq b\norm{f}_{S^p_{r_1}}$, and therefore in the sequel we will assume that $r=1$.
\end{notation}

\begin{definition}\label{stepanov alm per}
Let $p \in [1, +\infty)$. A function $f\in L^{p}_{\text{loc}}(\mathbb{R})$ is called \emph{$S^p$-almost periodic} if for any $\ep>0$ the set of all $(S^p,\ep)$-almost periods of $f$, defined as $S^pE\{\ep,f\} \zdef \bset{\tau \in \mathbb R}{\norm{f^{\tau}-f}_{S^p_1} <\varepsilon}$, is relatively dense.
\end{definition}

\begin{remark}\label{rem:AP_sap}
Clearly, any uniformly almost periodic function is $S^p$-almost periodic for every $p\in [1,+\infty)$. However, the continuous function $f \colon \mathbb R \to \mathbb R$ defined by the formula
\[
 f(t) = \sin \biggl(\frac{1}{2+\cos(t)+\cos(\sqrt{2}t)}\biggr) \qquad \text{for $t \in \mathbb R$},
\]
is an example of a $S^1$-almost periodic function which is not uniformly almost periodic (see for example~\cite{Levitan}*{p.~212} or~\cite{Stoinski}*{Example~2.2, p.~56}). Another simple example of a $S^p$-almost periodic function with $p \in [1,+\infty)$, which is not uniformly almost periodic, is given by the following formula $f(t)=a(-1)^{[\lambda t -\gamma]}$, where $a, \lambda > 0$, $\gamma \in \mathbb R$ and $[\,\cdot\,]$ denotes the \emph{entier} function. Finally, let us add that if a $S^p$-almost periodic function is uniformly continuous, then it is Bohr almost periodic (see~\cite{corduneanu}*{Theorem~6.16, p.~174} or~\cite{bezykowicz}*{Theorem, p.~81}).
\end{remark}

\begin{definition}
Let $p \in [1,+\infty)$. A function $f \in L^p_{\text{loc}}(\mathbb R)$ is said to be \emph{$S^p$-bounded} if $\norm{f}_{S^p_1} < +\infty$.
\end{definition}

\begin{remark}\label{rem:sp_bounded}
It is known that every $S^p$-almost periodic function, where $p \in [1,+\infty)$, is $S^p$-bounded (see~\cite{Stoinski}*{Theorem~2.1} or~\cite{Levitan}*{Theorem~5.2.2}). 
\end{remark}

\begin{remark}\label{rem:s_p_s_1}
The space $S^p(\mathbb R)$ of all $S^p$-almost periodic functions (equivalence classes) endowed with the norm $\norm{\cdot}_{S^p_1}$ is a Banach space (see~\cite{andres}). This space can also be obtained as the closure of the set of all generalized trigonometric polynomials in the Banach space $\set{f \in L^p_{\text{loc}}(\mathbb R)}{\norm{f}_{S^p_1}<+\infty}$ with respect to the norm $\norm{\cdot}_{S^p_1}$. Let us add that different values of $r>0$ give rise to different norms $\norm{\cdot}_{S^p_r}$ on $S^p(\mathbb R)$, but all of them generate the same topology. Finally, let us recall that if $f$ is a $S^{p_1}$-almost periodic function, then it is also $S^{p_2}$-almost periodic for $p_2 \leq p_1$ (for more details see for example~\cites{andres,bezykowicz}).
\end{remark}

Before we recall the definition of a function almost periodic  in the Lebesgue measure we need to introduce the following 

\begin{notation}
For $\eta>0$ and $f, g \in L^0(\mathbb R)$ let $D(\eta; f, g)\zdef \sup_{u \in \mathbb{R}}\mu\bigl(\set{t \in [u, u+1]}{\abs{f(t)-g(t)}\geq \eta}\bigr)$.
\end{notation}

\begin{remark}\label{lem:szacowanie_mu}
Clearly, if $f,g \in L^0(\mathbb R)$, then for every $\eta >0$ we have
\[
 D(\eta; f,g) \leq  2\sup_{z \in \mathbb Z} \mu\bigl(\set{t \in [z,z+1]}{\abs{f(t)-g(t)}\geq \eta}\bigr) \leq 2D(\eta; f,g).
\]
\end{remark}

\begin{definition}[\cites{St1, S}]\label{def4}
A function $f\in L^0(\mathbb R)$ is said to be \emph{almost periodic in the Lebesgue measure~$\mu$} (or \emph{$\mu$-almost periodic}) if for arbitrary numbers $\eps, \eta>0$ the set of all $(\eps, \eta)$-almost periods of $f$, defined as $\mu E\{\eps, \eta, f\}\zdef \bset{\tau \R}{D(\eta; f^{\tau},f) \leq \eps}$, is relatively dense.
\end{definition}

It turns out that the set of all $(\eps, \eta)$-almost periods of a $\mu$-almost periodic function is, in a sense, `quite big' as evidenced by the following proposition, whose proof we omit, since it is technical and follows from~\cite{St1}*{Lemma, p.~195} along the same lines as, for example,~\cite{Levitan}*{Theorem~5.2.4} or~\cite{Stoinski}*{Theorem~2.2}.

\begin{proposition}\label{prop:rel_dens_e_z}
If $f$ is $\mu$-almost periodic, then for every $\varepsilon, \eta >0$ the set $\mu E\{\varepsilon, \eta, f\} \cap \mathbb Z$ is relatively dense.
\end{proposition}

\begin{remark}\label{rem:rel_dens_e_z}
A result analogous to Proposition~\ref{prop:rel_dens_e_z} in the case of $S^1$-almost periodic functions is also true.
\end{remark}

\begin{remark}\label{rem:mu_S}
It can be easily shown that every $S^p$-almost periodic function is $\mu$-almost periodic. On the other hand, the function $f \colon \mathbb R \to \mathbb R$ given by
\[
 f(t) = \frac{1}{2+\cos(t)+\cos(\sqrt{2}t)} \qquad \text{for $t \in \mathbb R$},
\]
is an example of a continuous (and thus locally integrable) $\mu$-almost periodic function which is not $S^p$-almost periodic for any $p \in [1,+\infty)$ (for more details see~\cite{B-N}). 
However, if a $\mu$-almost periodic function is (essentially) bounded, then it is $S^p$-almost periodic for every $p \in [1,+\infty)$ (see~\cite{St1}*{Theorem~7} or~\cite{Stoinski}*{Theorem~4.11}).
\end{remark}

In the sequel we will also need the notion of a $D$-convergence.

\begin{definition}[\cite{St1}]\label{def6}
A sequence $(f_n)_{n \in \mathbb N}$, where $f_n \in L^0(\mathbb R)$ for $n\N$, is said to be \emph{$D$-convergent} to a function $f \in L^0(\mathbb R)$, if the following condition is satisfied 
\[
\forall \eps>0 \quad \forall \eta>0 \quad \exists N\N \quad \forall n>N \quad D(\eta; f_n, f)<\eps. 
\]
The function $f$ is said to be the \emph{$D$-limit} of the sequence $(f_n)_{n \in \mathbb N}$. 
\end{definition}

\begin{remark}\label{rem:D_zbieznosc2}
If $f \in L^0(\mathbb R)$ is the $D$-limit of a $D$-convergent sequence of $\mu$-almost periodic functions, then $f$ is $\mu$-almost periodic (see~\cite{St1}*{Theorem~6}).
\end{remark}

\begin{remark}\label{rem:D_zbieznosc}
It turns out that the $D$-convergence is metrizable. Indeed, in~\cite{St1} Stoi\'nski proved that a sequence $(f_n)_{n \in \mathbb N}$, where $f_n \in L^0(\mathbb R)$ for $n\N$, is $D$-convergent to $f \in L^0(\mathbb R)$ if and only if $\lim_{n \to \infty} \absp{f_n-f}=0$, where
\[
 \absp{f}\zdef \sup_{u \in \mathbb R} \int_u^{u+1} \dfrac{\abs{f(s)}}{1+\abs{f(s)}}\textup ds.
\]
In the same paper it was also shown (although not explicitly stated) that the functional $\absp{\cdot}$ is a complete $F$-norm on the vector space $M(\mathbb R)$ of all $\mu$-almost periodic functions. 
\end{remark}

\section{Mean value}
\label{sec:3}

In this section, we recall the concept of the mean value of an almost periodic function and we discuss its properties. Special attention is paid to the mean value of $\mu$-almost periodic functions.

Before passing to further considerations let us recall the following

\begin{definition}
Let $f \in L^1_{\text{loc}}(\mathbb R)$. The limit 
\[
 \mathcal{M}\{f\}:=\lim_{T \to +\infty} \frac{1}{T} \int_{0}^{T} f(s) \textup ds
\]
(whenever it exists) is called the \emph{mean value} of the function $f$.
\end{definition}

\begin{remark}\label{rem:mean_value_existence}
If $f$ is an almost periodic function in the sense of Bohr or Stepanov (with any $ 1 \leq p < +\infty$), then the mean value $\mathcal{M}\{f\}$ exists and is finite (see for example~\cite{Zaidman}*{Theorem~1, p.~85} or~\cite{bezykowicz}*{Theorem on p.~12 and Corollary~1 on p.~93}).
\end{remark}

The following example shows that the mean value of a continuous $\mu$-almost periodic function may not exist.

\begin{example}\label{ex:muap_nonexistence}
Let $A_n=4^n\mathbb Z+2^n$, $B_n=A_{2^n}$ and $a_n=(n+1)^2\cdot 2^{2^n}$ for $n\in \mathbb N$. Moreover, for a given $n \in \mathbb N$ let the function $f_n \colon \mathbb R\to \mathbb R$ be defined by the following formula 
\[
 f_n (x)= \begin{cases}
        a_nx-a_n\bigl(z+\tfrac{1}{2}-\frac{1}{n+1}\bigr)& \quad \text{for $x\in \bigl[z+\frac{1}{2}-\frac{1}{n+1}, z+\frac{1}{2}\bigr), z\in B_{n}$},\\[2mm]
         -a_nx+a_n\bigl(z+\tfrac{1}{2}+\frac{1}{n+1}\bigr)  & \quad \text{for $x\in \bigl[z+\frac{1}{2}, z+\frac{1}{2}+\frac{1}{n+1}\bigr), z\in B_{n}$},\\[2mm]
0& \quad \text{for other $x\in \mathbb R$}.
       \end{cases}
\]
It can be shown that the function $f \colon \mathbb R \to \mathbb R$ given by
\[
  f(x)=\sum_{n=1}^{\infty}f_n(x) \qquad \text{for $x \in \mathbb R$} 
\]
is continuous and $\mu$-almost periodic (cf.~\cite{B-N}*{Example 8} and Example~\ref{ex:mu_no_mu} below). 

For every $n \in \mathbb N$ we have
\[
\frac{1}{2^{2^n}}\int_{0}^{2^{2^n}}f_n(x)\textup dx=0 \quad \text{and}\quad \frac{1}{2^{2^n}+1}\int_{0}^{2^{2^n}+1}f_n(x)\textup dx=\frac{2^{2^n}}{2^{2^n}+1}.
\]
On the other hand, since the function $f_n$ is $4^{2^n}$-periodic, for $T\geq 4^{2^n}$ we have
\[
\frac{1}{T}\int_{0}^Tf_n(x)\textup dx\leq \frac{1}{[\frac{T}{4^{2^n}}]\cdot 4^{2^n}}\int_0^{\bigl(\bigl[\frac{T}{4^{2^n}}\bigr]+1\bigr)\cdot 4^{2^n}}f_n(x)\textup dx=\frac{[\frac{T}{4^{2^n}}]+1}{[\frac{T}{4^{2^n}}]}\cdot\frac{2^{2^n}}{4^{2^n}}\leq \frac{2}{2^{2^n}}.
\]
If $1 \leq k\leq n-1$, then $2^{2^n}\geq 4^{2^k}$. Therefore,
\[
\frac{1}{2^{2^{n}}}\int_0^{2^{2^{n}}}f(x)\textup dx=\sum_{k=1}^{n-1}\frac{1}{2^{2^{n}}}\int_0^{2^{2^{n}}}f_k(x)\textup dx\leq \sum_{k=1}^{n-1}\frac{2}{2^{2^k}}\leq 2\sum_{k=1}^{\infty}\frac{1}{2^{2k}}= \frac{2}{3} \quad \text{for $n \geq 2$}.
\]
Finally, for $n\N$ we also have
\[
\frac{1}{2^{2^{n}}+1}\int_0^{2^{2^{n}}+1}f(x)\textup dx\geq \frac{1}{2^{2^{n}}+1}\int_0^{2^{2^{n}}+1}f_n(x)\textup dx= \frac{2^{2^n}}{2^{2^{n}}+1}.
\]
Hence the mean value $\mathcal M\{f\}$ does not exist.
\end{example}

Observe that the function $f$ from Example~\ref{ex:muap_nonexistence} is not $S^1$-bounded. Let us also add that it is possible to construct an example of a continuous $\mu$-almost periodic function $f$ which is not $S^1$-bounded and whose mean value exists but is infinite, that is, $\mathcal M\{f\}=+\infty$. Therefore, a natural question arises whether every locally integrable $\mu$-almost periodic function which is not $S^1$-bounded fails to have finite mean value. The answer to this question is provided by the following example. 

\begin{example}\label{ex:3.4}
Let $A_n=2\cdot 3^n\mathbb Z-3^n$ for $n\in \mathbb N$. Let us put
\[
 f_n (x)= \begin{cases}
        n^2& \text{for $x\in [z, z+\frac{1}{n}), z\in A_{n}$},\\[2mm]
        0& \text{for other $x\in \mathbb R$}.
       \end{cases}
\]
As in Example~\ref{ex:muap_nonexistence} define a locally integrable $\mu$-almost periodic function $f \colon \mathbb R \to \mathbb R$ by the formula
\[
 f(x)=\sum_{n=1}^{\infty}f_n(x) \qquad \text{for $x \in \mathbb R$}
\]
(cf. also~\cite{B-N}*{Example 8} and Example~\ref{ex:mu_no_mu} below). Note that, since 
\[
\int_{z}^{z+1}f(x)\textup dx\geq \int_z^{z+1}f_n(x)\textup dx=n\quad \text{for $z\in A_n$},
\]
the function $f$ is not $S^1$-bounded, and hence cannot be $S^1$-almost periodic (see Remark~\ref{rem:sp_bounded}). However, as we will show below, its mean value exists and 
\[
 \mathcal{M}\{f\}=\lim_{T\to +\infty}\frac{1}{T}\int_0^T f(x) \textup dx=\sum_{n=1}^{\infty}\frac{n}{2\cdot 3^n}.
\]

First, observe that for every $n \in \mathbb N$ the function $f_n$ is $(2\cdot 3^n)$-periodic, and therefore
\[
  \lim_{T\to +\infty}\frac{1}{T}\int_0^T f_n(x) \textup dx= \frac{1}{2\cdot 3^n}\int_0^{2\cdot 3^n}f_n(x) \textup dx=\frac{n}{2\cdot 3^n}
\]
(cf.~\cite{Zaidman}*{Remark, p.~88}). Furthermore, for every $T>0$ we have
\[
 \frac{1}{T}\int_0^T f_n(x) \textup dx\leq \frac{n}{3^n}.
\]
Indeed, if $0 < T \leq 3^n$, then
\[
 \frac{1}{T}\int_0^T f_n(x) \textup dx=0\leq  \frac{n}{3^n}; 
\]
if $3^n<T<2\cdot 3^n$, then
\[
 \frac{1}{T}\int_0^T f_n(x)\textup dx\leq \frac{1}{3^n}\int_0^{2\cdot 3^n}f_n(x) \textup dx=\frac{n}{3^n};
\]
and finally, if $T\geq 2\cdot 3^n$, then
\[
 \frac{1}{T}\int_0^T f_n(x) \textup dx \leq \frac{1}{2\cdot 3^n \bigl[\frac{T}{2\cdot 3^n}\bigr]}\int_0^{2\cdot 3^n\bigl(\bigl[\frac{T}{2\cdot 3^n}\bigr]+1\bigr)}f_n(x) \textup dx=  \frac{\bigl[\frac{T}{2\cdot 3^n}\bigr]+1}{\bigl[\frac{T}{2\cdot 3^n}\bigr]}\frac{n}{2\cdot 3^n}\leq \frac{n}{3^n}.
\]
Given $\eps>0$ let $k\in \mathbb N$ be such that
\[
 \sum_{n=k+1}^{\infty}\frac{n}{3^n}<\frac{1}{3}\eps,
\]
and choose $T_0>0$ such that for $T \geq T_0$ we have
\[
 \bgabs{\frac{1}{T}\int_{0}^T\sum_{n=1}^k f_n(x) \textup dx-\sum_{n=1}^k\frac{n}{2\cdot3^n}}<\frac{1}{3}\eps.
\]
Then, for $T \geq T_0$ we have
\begin{multline*}
 \bgabs{\frac{1}{T}\int_{0}^Tf(x) \textup dx-\sum_{n=1}^{\infty}\frac{n}{2\cdot3^n}}=\bgabs{\frac{1}{T}\int_{0}^T\sum_{n=1}^{\infty}f_n(x) \textup dx-\sum_{n=1}^{\infty}\frac{n}{2\cdot3^n}}\leq\\[2mm] 
 \leq \bgabs{\frac{1}{T}\int_{0}^T\sum_{n=1}^k f_n(x) \textup dx-\sum_{n=1}^k\frac{n}{2\cdot3^n}}+\frac{1}{T}\int_{0}^T\sum_{n=k+1}^{\infty}f_n(x) \textup dx+\sum_{n=k+1}^{\infty}\frac{n}{2\cdot3^n}\leq\\[2mm] 
 \leq \frac{1}{3}\eps+\sum_{n=k+1}^{\infty}\frac{n}{3^n}+\sum_{n=k+1}^{\infty}\frac{n}{3^n}\leq \eps
\end{multline*}
(let us note that we could change the order of summation and integration in the above estimate, since on each interval $[0,T]$ only finitely many functions $f_n$ do not vanish; cf.~also~\cite{HSt}*{Theorem~12.21}). This shows that 
\[
 \mathcal{M}\{f\}=\sum_{n=1}^{\infty}\frac{n}{2\cdot 3^n}.
\]
\end{example}

Now we are going to provide a sufficient condition guaranteeing that the mean value of a locally integrable $\mu$-almost periodic functions exists. However, first we need the following

\begin{definition}
For $f\colon \mathbb R \to \mathbb R$ and $N>0$ we define the \emph{truncated function} $f_N$ corresponding to the function $f$ by the formula $f_N(x)=\max\bigl(-N, \min(f(x),N)\bigr)$ for $x \in \mathbb R$. 
\end{definition}

\begin{theorem}\label{thm:muap_mean}
Let $f \in L^1_{\text{loc}}(\mathbb R)$ be a non-negative \textup(almost everywhere\textup) $\mu$-almost periodic function. If for every $\eps>0$ there exist positive numbers $T_0$ and $N_0$ such that
 \[
  \frac{1}{T}\int_{0}^T \bigl(f(x)-f_{N_0}(x)\bigr)\textup dx <\eps \qquad \text{for $T \geq T_0$},
 \]
then $\mathcal M\{f\}$ exists and is finite, and moreover $\mathcal M\{f\}=\lim_{N\to +\infty}\mathcal M\{f_N\}$.
\end{theorem}

Before we proceed to the proof of Theorem~\ref{thm:muap_mean}, several remarks are in order.

\begin{remark}\label{rem:muap_mean}
Let us observe that if $f$ is a $\mu$-almost periodic function, then the truncated function $f_N$ is $\mu$-almost periodic for every $N>0$ (cf.~\cite{St1}*{the proof of Theorem~9} and see~\cite{St2}*{p.~172}), and since it is also bounded, we infer that $f_N$ is $S^p$-almost periodic for every $p \in [1,+\infty)$ (cf.~Remark~\ref{rem:mu_S}). In particular, the mean value $\mathcal M\{f_N\}$ exists and is finite (see Remark~\ref{rem:mean_value_existence}). If, in addition, the function $f$ is assumed to be almost everywhere non-negative, then the mapping $N \mapsto \mathcal M\{f_N\}$ is non-decreasing and the limit $\lim_{N\to +\infty}\mathcal M\{f_N\}$ exists (we do not exclude here that the limit is equal to $+\infty$).
\end{remark}

\begin{proof}[Proof of Theorem~\textup{\ref{thm:muap_mean}}]
In view of the assumption, for $\eps=1$ and $N\geq N_0$ we have 
\[
  \frac{1}{T}\int_{0}^T \bigl(f_N(x)-f_{N_0}(x)\bigr)\textup dx\leq \frac{1}{T}\int_{0}^T \bigl(f(x)-f_{N_0}(x)\bigr)\textup dx <1, \qquad \text{whenever $T \geq T_0$}.
 \]
Thus, by Remark~\ref{rem:muap_mean}, we infer that the limit $m \zdef \lim_{N\to +\infty}\mathcal M\{f_N\}$ exists and is finite. 

Now, we will show that $\mathcal{M}\{f\}=m$. Given $\eps>0$ there exists $N_1>0$ such that $0\leq m-\mathcal M\{f_{N}\}< \frac{1}{3}\eps$ for $N\geq N_1$. Furthermore, in view of the assumption, there exist $T_1, N_2>0$ such that
\[
 \frac{1}{T}\int_{0}^T \bigl(f(x)-f_{N}(x)\bigr) \textup dx\leq  \frac{1}{T}\int_{0}^T \bigl(f(x)-f_{N_2}(x)\bigr) \textup dx <\frac{1}{3}\eps \qquad \text{for $T\geq T_1$ and $N\geq N_2$}.
\]

Let $N_3=\max{(N_1, N_2)}$. Then, since the function $f_{N_3}$ is $S^1$-almost periodic (cf.~Remark~\ref{rem:muap_mean}), the mean value $\mathcal M\{f_{N_3}\}$ exists and is finite, and therefore there is $T_2>0$ such that 
\[
\bgabs{\frac{1}{T}\int_{0}^T f_{N_3}(x) \textup dx - \mathcal M\{f_{N_3}\}} < \frac{1}{3}\eps \qquad \text{for $T\geq T_2$}.
\]
Hence, for $T\geq \max{(T_1, T_2)}$ we have
\begin{equation*}
 \bgabs{\frac{1}{T}\int_{0}^T\! f(x) \textup dx-m}  \leq \frac{1}{T}\int_{0}^T\! \bigl(f(x)-f_{N_3}(x)\bigr) \textup dx+ \bgabs{\frac{1}{T}\int_{0}^T\! f_{N_3}(x) \textup dx-\mathcal M\{f_{N_3}\}}+m-\mathcal M\{f_{N_3}\}<\eps,
\end{equation*}
which shows that $\mathcal M\{f\}=m$ and ends the proof.
\end{proof}

\begin{remark}
Let us note that if a locally integrable function $f \colon \mathbb R \to \mathbb R$ has finite mean value $\mathcal M\{f\}$, then for every $\alpha \in \mathbb R$ we have
\begin{equation}\label{eq:srednia_alpha}
\lim_{T\to +\infty}\frac{1}{T}\int_\alpha^{\alpha+T}f(x) \textup dx=\mathcal M\{f\}.
\end{equation}
It is also known that for functions almost periodic in the sense of Bohr or Stepanov (with any $p \in [1,+\infty)$), the limit in~\eqref{eq:srednia_alpha} exists uniformly in $\alpha \in \mathbb R$ (see~\cite{Levitan}*{Theorem~1.3.2  and Theorem~5.6.2} or~\cite{Stoinski}*{Theorem~1.9 and Theorem~2.16}). However, this result is no longer true, if we consider functions which are $\mu$-almost periodic; to see this it suffices to consider the function $f$ defined in Example~\ref{ex:3.4}.
\end{remark}

Now, we will proceed to the study of further properties of the mean value, which will be needed in the sequel.

\begin{proposition}[cf.~\cite{wmjs1}*{Lemma~3.7}]\label{o sredniej dla nieujemnej}
Let $f \in S^1(\mathbb R)$ be almost everywhere non-negative. Then $\mathcal{M}\{f\}=0$ if and only if $f(t)=0$ for almost all $t\in\mathbb{R}$.
\end{proposition}

\begin{proof}
It is obvious that $f\equiv 0$ a.e. implies $\mathcal{M}\{f\}=0$. The remaining part was proved in~\cite{wmjs1}.
\end{proof}

Clearly, if in Proposition~\ref{o sredniej dla nieujemnej} the function $f$ is uniformly almost periodic, then the equality $\mathcal{M}\{f\}=0$ implies that $f(t)=0$ for every $t \in \mathbb R$.

In the case of $\mu$-almost periodic functions we have the following

\begin{theorem}\label{thm:zero}
Let $f \colon \mathbb R \to \mathbb R$ be a locally integrable $\mu$-almost periodic function which is almost everywhere non-negative. Then  $\mathcal{M}\{f\}=0$ if and only if $f(t)=0$ for almost all $t\in\mathbb{R}$.
\end{theorem}

\begin{proof}
Since the sufficiency part is obvious, let us assume that $\mathcal{M}\{f\}=0$ and suppose that $f$ does not vanish almost everywhere. Then, there exist a point $u\R$, together with positive numbers $\eps, \eta$ and a Lebesgue measurable set $A\subseteq [u, u+1]$ with $\mu(A)=\eps$ such that  $f(x)\geq \eta$ for a.e. $x\in A$. For every $n\N$ choose $\tau_n \in (2(n-1)\omega-u, 2(n-1)\omega-u+\omega)\cap \mu E\{\frac{\eps}{2}, \frac{\eta}{2}, f\}$, where $\omega$ is a number which characterizes the relative density of the set $\mu E\{\frac{\eps}{2}, \frac{\eta}{2},f\}$ (clearly, we may assume that $\omega>1$). Then
\[
\mu\bigl(\set{x \in [u, u+1]}{\abs{f(x+\tau_n)-f(x)}<\tfrac{\eta}{2}}\bigr)\geq 1-\frac{\eps}{2},
\]
and thus
\[
\mu\bigl(\set{x \in A}{f(x+\tau_n)\geq \tfrac{\eta}{2}}\bigr)\geq \mu\bigl(\set{x \in A}{\abs{f(x+\tau_n)-f(x)}< \tfrac{\eta}{2}}\bigr)\geq \frac{\eps}{2}.
\]
So, we have
\[
 \int_{2(n-1)\omega}^{2n\omega}f(x) \textup dx\geq \frac{\eps \eta}{4} \qquad \text{for $n\N$},
\]
and hence
\[
\frac{1}{2n\omega}\int_{0}^{2n\omega}f(x) \textup dx\geq \frac{\eps \eta}{8\omega} \qquad \text{for $n \in \mathbb N$}.
\]
This contradicts the fact that $\mathcal M\{f\}=0$.
\end{proof}

\begin{remark}
The difference between non-negative $\mu$-almost periodic functions and non-negative $S^p$-almost periodic functions is the following: for $S^p$-almost periodic functions we have $\mathcal M\{f\}=0$ or $\mathcal M\{f\}>0$, while for $\mu$-almost periodic functions from the negation of the condition $\mathcal M\{f\}=0$ it does not follow that $\mathcal M\{f\}>0$, since the mean value may not exist.
\end{remark}

\begin{example}
Let us note that, in general, $\mathcal{M}\{f\}=0$ does not imply that $f(t)=0$ a.e. on $\mathbb R$, even if $f$ is non-negative. Let us take for instance the function $f\colon \mathbb R\to \mathbb R$ defined by the formula $f(t)=\ue^{-t^2/2}$ for $t \in \mathbb R$. Then, for every $T>0$, we have
\[
 \frac{1}{T}\int_{0}^T f(u)\; \textup du \leq \frac{1}{T}\int_{\mathbb R} f(u)\; \textup du = \frac{\sqrt{2\pi}}{T}, 
\]
which shows that $\mathcal{M}\{f\}=0$, even though $f(t)>0$ for every $t\in \mathbb{R}$.
\end{example}

\begin{remark}
Let us observe that the mean value $\mathcal{M}$ is a continuous functional on the space $AP(\mathbb R)$ or $S^p(\mathbb R)$ with $p \in [1,+\infty)$, which means that $\mathcal{M}\{f_n\}\to\mathcal{M}\{f\}$ if $f_n\to f$ in the corresponding almost-periodic norm, since $\abs{\mathcal M\{f-g\}} \leq \mathcal M\{\abs{f-g}\}\leq \norm{f-g}_{\infty}$ for $f, g\in AP(\mathbb R)$ and $\abs{\mathcal M\{f-g\}} \leq \mathcal M\{\abs{f-g}\} \leq \norm{f-g}_{S^p_1}$ for $f, g \in S^p(\mathbb R)$.
\end{remark}

The next example shows, however, that, if a sequence $(f_n)_{n \in \mathbb N}$ of $\mu$-almost periodic functions is $D$-convergent to $f$, then it may happen that $\mathcal M\{f_n\}\not\to \mathcal M\{f\}$, even if $\mathcal M\{f\}$ exists and is finite.

\begin{example}
For every $n \in \mathbb N$ let
\[
 f_n (x)= \begin{cases}
        n& \text{for $x\in [z, z+\frac{1}{n}), z\Z$},\\[2mm]
        0& \text{for other $x\R$},
       \end{cases}
\]
and let us note that the functions $f_n$ are locally integrable and 1-periodic, and thus $\mu$-almost periodic.

It is easy to show that the sequence $(f_n)_{n \in \mathbb N}$ is $D$-convergent to the zero function $f$. However, $\mathcal M\{f_n\}\not\to \mathcal M\{f\}$, since $\mathcal M\{f_n\}=1 $ for $n\N$ and $\mathcal M\{f\}=0$.
\end{example}

\section{Firing map and its displacement}
\label{sec:firing_map}

The following section is devoted to the study of the firing map and the displacement map corresponding to the LIF model~\eqref{lif4}--\eqref{lif_spiking}. Starting in Subsection~\ref{subsection:4.1} with the discussion of some general properties of the above-mentioned maps, we then move to the investigation of the firing map and its displacement for the LIF model driven by Stepanov and $\mu$-almost periodic functions.

\subsection{General properties of the firing map and its displacement}
\label{subsection:4.1}

Since throughout the rest of the paper we will consider leaky integrate-and-fire models with a locally integrable almost periodic input, first we need to rewrite the definition of the firing map for that setting.

\begin{definition}\label{firingdlalokalnie}
Let $f \in L^1_{\text{loc}}(\mathbb R)$. The \emph{firing map} $\Phi$ corresponding to the system~\eqref{lif4}--\eqref{lif_spiking} is defined as
\begin{equation*}
\Phi(t) \zdef \inf\bgset{t_*>t}{\ue^{\sigma t}\leq \int_t^{t_*}\bigl(f(u)-\sigma\bigr)\ue^{\sigma u}\; \textup du}, \qquad t \in \mathbb R.
\end{equation*}
\end{definition}

\begin{remark}
Let us observe that if by $x(\cdot; t,0)$ we denote the solution of~\eqref{lif4} originating from the point $(t,0)$, then, equivalently, the value of the firing map at $t$ may be defined by the formula $\Phi(t) = \inf\set{t_\ast > t}{x(t_\ast;t,0) \geq 1}$. Therefore, roughly speaking, the firing map assigns to each point $t \in \mathbb R$ the time $\Phi(t)$ at which the trajectory of~\eqref{lif4} originating from $(t,0)$ reaches the threshold.
\end{remark}

\begin{example}\label{ex:phi_f_piecewise_constant}
Let $\sigma=1$ and let us consider the LIF model~\eqref{lif4}--\eqref{lif_spiking} driven by the locally integrable $2$-periodic function $f \colon \mathbb R \to \mathbb R$ given by
\[
 f(t) = \begin{cases}
         2 \quad & \text{for $t \in [2k,2k+1)$, $k \in \mathbb Z$},\\
				 1 \quad & \text{for $t \in [2k+1,2k+2)$, $k \in \mathbb Z$}.
				\end{cases}
\]
It can be checked that the firing map $\Phi$ corresponding to such a model has the formula
\[
 \Phi(t) = \begin{cases}
            \ln(2e^t) \quad & \text{for $t \in [2k,2k+1-\ln{2}]$, $k \in \mathbb Z$},\\
						\ln(2e^t + e^{2k+2} - e^{2k+1}) \quad & \text{for $t \in (2k+1-\ln{2}, 2k+1)$, $k \in \mathbb Z$},\\
						\ln(e^t + e^{2k+2}) \quad & \text{for $t \in [2k+1,2k+2)$, $k \in \mathbb Z$}.
						\end{cases}
\]
\end{example}

It may happen that the firing map $\Phi$ is not well-defined for every $t \in \mathbb R$, meaning that for some $\tau \in \mathbb R$ the set appearing in the definition of $\Phi(\tau)$ is empty. However, we have the following known result which describes necessary and sufficient conditions for $\Phi$ to be well-defined for every $t \in \mathbb R$; for Readers' convenience we will provide its proof, since the proof presented in~\cite{MS} contains a minor gap.

\begin{proposition}[\cite{MS}*{Lemma~2.2}]\label{poprawneokr eslenie}
Let $f \in L^1_{\text{loc}}(\mathbb R)$. The firing map $\Phi$ corresponding to the system~\eqref{lif4}--\eqref{lif_spiking} is well-defined for every $t \in \mathbb R$ if and only if
\begin{equation}\label{warunek}
\limsup_{t\to+\infty}\int_0^t\bigl(f(u)-\sigma\bigr)\ue^{\sigma u}\; \textup du=+\infty.
\end{equation}
\end{proposition}

Before we pass to the proof of Proposition~\ref{poprawneokr eslenie}, let us observe that if the value $\Phi(t)$ is well-defined for some $t \in \mathbb R$, then it has to satisfy the implicit equation
\begin{equation}\label{implicitnafiringdlalif}
\ue^{\sigma t}=\int_{t}^{\Phi(t)} \bigl(f(u)-\sigma\bigr)\ue^{\sigma u}\; \textup du.
\end{equation}

\begin{proof}[Proof of Proposition~\textup{\ref{poprawneokr eslenie}}]
Suppose that the condition~\eqref{warunek} is satisfied and fix $t_0\in\mathbb{R}$. Then $\limsup_{t\to+\infty}\int_{t_0}^t \bigl(f(u)-\sigma\bigr)\ue^{\sigma u}\; \textup du=+\infty$, and hence there exists $t_* > t_0$ such that $\int_{t_0}^{t_*}\bigl(f(u)-\sigma\bigr)\ue^{\sigma u}\; \textup du\geq \ue^{\sigma t_0}$. Consequently, $\Phi(t_0)$ is defined.

Now, let us assume that $\Phi \colon \mathbb{R}\to\mathbb{R}$ is well-defined for every $t\in\mathbb{R}$. In particular, by~\eqref{implicitnafiringdlalif}, for $t=0$ and every $n \in \mathbb N$ we have
\[
 \int_0^{\Phi^n(0)} \bigl( f(u) - \sigma\bigr) \ue^{\sigma u}\; \textup du = \sum_{i=1}^n \int_{\Phi^{i-1}(0)}^{\Phi^i(0)} \bigl( f(u) - \sigma\bigr) \ue^{\sigma u}\; \textup du = \sum_{i=1}^n \ue^{\sigma \Phi^{i-1}(0)} \geq n;
\]
here $\Phi^n$ denotes the $n$-th iterate of $\Phi$ and, by definition, we set $\Phi^0(0)=0$. Moreover, let us observe that the increasing sequence $\bigl(\Phi^n(0)\bigr)_{n \in \mathbb N}$ is unbounded, since otherwise we would have $n \leq \int_0^a \abs{f(u)-\sigma}\ue^{\sigma u}\; \textup du <+\infty$ for $n \in \mathbb N$ and some $a \in (0,+\infty)$ which is independent of $n$, and a contradiction with the local integrability of $f$ would follow. Thus $\lim_{n \to \infty} \Phi^n(0)=+\infty$ and $\lim_{n\to\infty}\int_0^{\Phi^n(0)}\bigl(f(u)-\sigma\bigr)\ue^{\sigma u}\; \textup du=+\infty$, which proves the claim. 
\end{proof}

\begin{remark}\label{rem:Phi_w_d_2}
Let us note that a slight modification of the proof of Proposition~\ref{poprawneokr eslenie} shows that if $f \in L^1_{\text{loc}}(\mathbb R)$, then the firing map $\Phi$ corresponding to the LIF model~\eqref{lif4}--\eqref{lif_spiking} is well-defined on $\mathbb R$ if and only if for some $t \in \mathbb R$ all the iterates $\Phi^n(t)$ are well-defined.
\end{remark}

 From Proposition~\ref{poprawneokr eslenie} we get the following corollary, which in the case of the PI model was stated in~\cite{wmjs1}.

\begin{corollary}\label{cor:l1-phi}
Let $f \in L^1_{\text{loc}}(\mathbb R)$ and suppose that the mean value $\mathcal M\{f\}$ exists\footnote{We do not exclude the case when the mean value $\mathcal M\{f\}$ is infinite.}. If $\mathcal M\{f\} >\sigma$, then the firing map $\Phi$ corresponding to the LIF model~\eqref{lif4}--\eqref{lif_spiking} is well-defined for every $t \in \mathbb R$.
\end{corollary}

\begin{proof}
For simplicity, let us set
\[
 \mathcal N(t) = \int_0^t \bigl(f(u)-\sigma\bigr)e^{\sigma u} \textup du \qquad \text{for $t \geq 0$}.
\]
It is easy to see that 
\begin{equation}\label{eq:calkiii}
 \mathcal N(t) = e^{\sigma t}\int_0^t\bigl( f(u) - \sigma\bigr) \textup du - \sigma\int_0^t\biggl(\int_0^u \bigl(f(w)-\sigma\bigr) \textup dw\biggr) e^{\sigma u} \textup du;
\end{equation}
indeed, it suffices to change the order of integration in the second integral on the right-hand side of~\eqref{eq:calkiii}. 

Now, we will show that for every $n \in \mathbb N$ there exists $t_n \geq n$ such that $\mathcal N(t_n)> n$. Suppose on the contrary that there is some $n$ such that 
\begin{equation}\label{eq:gronwall1}
 e^{\sigma t} \int_0^t\bigl(f(u)-\sigma\bigr) \textup du \leq  n + \sigma\int_0^t\biggl(\int_0^u \bigl(f(w)-\sigma\bigr) \textup dw\biggr) e^{\sigma u} \textup du \qquad \text{for all $t \geq n$}.
\end{equation}
Let $g(t):=e^{\sigma t} \int_0^t\bigl(f(u)-\sigma\bigr) \textup du$. Then~\eqref{eq:gronwall1} can be equivalently rewritten as
\[
 g(t) \leq n + \sigma\int_0^n g(u) \textup du + \sigma\int_n^t g(u)\textup du \qquad \text{for all $t \geq n$}. 
\]
Applying Gronwall's inequality (see, for example,~\cite{BS}*{Corollary~1.4}), we infer that
\[
 g(t) \leq \biggl(n+\sigma\int_0^n g(u) \textup du\biggr) e^{\sigma(t-n)} \qquad \text{for $t \geq n$}.
\]
Therefore,
\begin{equation}\label{eq:mean_III}
 \frac{1}{t} \int_0^t \bigl(f(u)-\sigma\bigr)\textup du \leq \frac{1}{t}e^{-\sigma n}\biggl[n+\sigma\int_0^n \biggl(\int_0^u \bigl(f(w)-\sigma\bigr) \textup dw\biggr) e^{\sigma u} \textup du\biggr] \qquad \text{for $t \geq n$}. 
\end{equation}
Passing to the limit in~\eqref{eq:mean_III} with $t \to +\infty$, yields $\mathcal M\{f\}-\sigma\leq 0$. This, however, leads to a contradiction. 

This shows that for every $n \in \mathbb N$ there exists $t_n \geq n$ such that $\mathcal N(t_n)> n$, and thus
\[
 \limsup_{t \to +\infty} \int_0^t \bigl(f(u)-\sigma\bigr)e^{\sigma u} \textup du = +\infty.
\]
To end the proof it suffices now to apply Proposition~\ref{poprawneokr eslenie}.
\end{proof}

\begin{remark}\label{rem:rosnacy}
Let us note that the firing map $\Phi$ corresponding to the system~\eqref{lif4}--\eqref{lif_spiking} may be not well-defined for every $t \in \mathbb R$ even if the function $f \in L^1_{\text{loc}}(\mathbb R)$ is such that $f-\sigma>0$ a.e. on $\mathbb R$; to see this it suffices to consider the function $f \colon \mathbb R \to \mathbb R$ given by $f(t)=\sigma+e^{-(\sigma+1)t}$. However, it should be observed that if $f \in L^1_{\text{loc}}(\mathbb R)$ is such that $f(t)-\sigma \geq 0$ for a.e. $t \in \mathbb R$ and $\Phi(t_0)$ is defined for some $t_0 \in \mathbb R$, then  $\Phi(t)$ is defined for every $t \leq t_0$. For functions $f$ such that the difference $f-\sigma$ is negative on some set of positive Lebesgue measure, the above claim may not hold; to see this it suffices to consider the PI model and the function $f \colon \mathbb R \to \mathbb R$ given by $f(t)=\frac{1}{2}\sin{t}$ for $t \in \mathbb R$, since then we have $\Phi(2\pi)=3\pi$, but $\int_{\frac{\pi}{2}}^t f(u)\textup du <1$ for every $t > \frac{\pi}{2}$, which shows that $\Phi(\frac{\pi}{2})$ is not well-defined.
\end{remark}

In our further considerations we will also need the following simple result on the monotonicity of the firing map (for a similar result for the PI model see~\cite{wmjs1}).

\begin{lemma}\label{lem:rosnie}
Suppose that $f \in L^1_{\text{loc}}(\mathbb R)$ is such that $f(t)-\sigma >  0$ for a.e. $t \in \mathbb R$ and suppose that $\Phi(t_0)$ is defined for some $t_0 \in \mathbb R$, where $\Phi$ is the firing map corresponding to the LIF model~\eqref{lif4}--\eqref{lif_spiking}. Then $\Phi(s)$ is well-defined for every $s<t_0$ and $\Phi(s)<\Phi(t_0)$.
\end{lemma}

\begin{proof}
First, let us note that in view of Remark~\ref{rem:rosnacy} $\Phi(s)$ is defined for every $s <t_0$.

Now, on the contrary, let us suppose that $\Phi(t_0) \leq \Phi(s)$. Then
\[
 \int_{t_0}^{\Phi(t_0)} \bigl(f(u)-\sigma\bigr)e^{\sigma(u-t_0)} \textup du = \int_s^{\Phi(s)} \bigl(f(u)-\sigma\bigr)e^{\sigma(u-s)} \textup du,
\]
and thus
\begin{align*}
 0 \leq \int_{\Phi(t_0)}^{\Phi(s)} \bigl(f(u)-\sigma\bigr) e^{\sigma(u-s)} \textup du & = \int_{t_0}^{\Phi(t_0)}\bigl(f(u)-\sigma\bigr) e^{\sigma(u-t_0)} \textup du -  \int_s^{\Phi(t_0)}\bigl(f(u)-\sigma\bigr) e^{\sigma(u-s)} \textup du\\[2mm]
 & \leq - \int_s^{t_0} \bigl(f(u)-\sigma\bigr) e^{\sigma(u-s)} \textup du < 0.
\end{align*}
This leads to a contradiction. Therefore, $\Phi(t_0) > \Phi(s)$.
\end{proof}

\begin{remark}
Similar result to Lemma~\ref{lem:rosnie} holds also in the case when $f(t)-\sigma \geq 0$ for a.e. $t \in \mathbb R$. Then, in the claim one needs to replace the strict inequality $\Phi(s)<\Phi(t_0)$ with $\Phi(s) \leq \Phi(t_0)$.
\end{remark}

Since it will be easier to formulate our next result in terms of the displacement map, we proceed now with the following

\begin{definition}\label{displacement map}
Let $f \in L^1_{\text{loc}}(\mathbb R)$ and let $\Phi$ be the firing map corresponding to the system~\eqref{lif4}--\eqref{lif_spiking}. The \emph{displacement map} $\Psi$ of the map $\Phi$ is defined as $\Psi(t) \zdef \Phi(t)-t$.
\end{definition}

Clearly, the displacement map is well-defined only for those $t$'s for which the value $\Phi(t)$ is well-defined. Then of course $\Psi(t)\geq 0$, since $\Phi(t)\geq t$ by definition. Roughly speaking, the value $\Psi(t)$ says how long we have to wait for the next firing if the previous firing was at time $t$.

Our next result describes sufficient conditions for the displacement map to be uniformly continuous. 

\begin{theorem}\label{lem:uniform_con_bound}
Let $\varsigma>0$ and assume that $f \in L^1_{\text{loc}}(\mathbb R)$ satisfies the following conditions\textup:
\begin{enumerate}[itemsep=2pt, label=\textup{(\roman*)}]
  \item\label{it:mu_s1} $\displaystyle \sup_{s \in \mathbb R} \int_s^{s+\delta} f(u) \textup du \to 0$ as $\delta \to 0^+$\textup;

  \item\label{it:mu_s2} $f(t) - \sigma \geq \varsigma$ for a.e. $t \in \mathbb R$.
\end{enumerate}
Then the displacement map $\Psi \colon \mathbb R \to \mathbb R$ corresponding to the LIF model~\eqref{lif4}--\eqref{lif_spiking} is bounded and uniformly continuous. Moreover, $\inf_{t \in \mathbb R}\Psi(t)>0$.
\end{theorem}

\begin{proof}
First, let us note that in view of the assumption~\ref{it:mu_s2} and Proposition~\ref{poprawneokr eslenie} the firing map $\Phi$ and the displacement map $\Psi$ are well-defined for every $t \in \mathbb R$. 

The boundedness of the displacement map $\Psi$ follows from the following estimate
\[
 e^{\sigma t} = \int_t^{\Phi(t)} \bigl(f(u)-\sigma\bigr) e^{\sigma u} \textup du \geq \varsigma e^{\sigma t} \bigl(\Phi(t)-t\bigr)=\varsigma e^{\sigma t}\Psi(t) \qquad \text{for every $t \in \mathbb R$}.
\]

Now, we are going to prove that $\Psi$ is uniformly continuous. We start with observing that, in view of the assumption~\ref{it:mu_s1}, for every $t \in \mathbb R$ the integral $\int_t^{t+1/\varsigma} f(u)\textup du$ can be estimated from above by a constant independent of $t$. Indeed, there exists $\delta_0>0$ such that
\[
 \sup_{s \in \mathbb R} \int_s^{s+\delta_0} f(u) \textup du \leq 1,
\]
and so, if $k \in \mathbb N$ is the smallest number such that $1/\varsigma \leq k \delta_0$, then
\[
 \int_t^{t+1/\varsigma} f(u) \textup du \leq \sum_{i=0}^{k-1} \int_{t+i\delta_0}^{t+(i+1)\delta_0} f(u) \textup du \leq k \cdot \sup_{s \in \mathbb R} \int_s^{s+\delta_0} f(u) \textup du \leq k \qquad \text{for every $t \in \mathbb R$}.
\]

Now, for a given $\varepsilon>0$ choose $\delta>0$ such that
\[
 \frac{1}{\varsigma} e^{\sigma \delta} \sup_{s \in \mathbb R} \int_s^{s+\delta} f(u)\textup du \leq \frac{1}{2}\varepsilon \qquad \text{and} \qquad \frac{1}{\varsigma}(e^{\sigma \delta}-1)k \leq \frac{1}{2}\varepsilon.
\]
Take arbitrary points $t, \tau \in \mathbb R$ such that $\abs{t-\tau}\leq \delta$. Without loss of generality, we may assume that $t\leq \tau$. Then $\Phi(t) \leq \Phi(\tau)$ and
\[
  \int_t^{\Phi(t)} \bigl(f(u)-\sigma\bigr)e^{\sigma(u-t)} \textup du = \int_{\tau}^{\Phi(\tau)} \bigl(f(u)-\sigma\bigr)e^{\sigma(u-\tau)} \textup du. 
\]
Thus
\[
 \int_{\Phi(t)}^{\Phi(\tau)} \bigl(f(u)-\sigma\bigr)e^{\sigma(u-\tau)} \textup du = \int_t^{\tau} \bigl(f(u)-\sigma\bigr)e^{\sigma(u-\tau)} \textup du + \int_{t}^{\Phi(t)}\! \bigl(f(u)-\sigma\bigr)\! \bigl( e^{\sigma(u-t)} - e^{\sigma(u-\tau)} \bigr) \textup du.
\]
Our aim is to estimate the above integrals. For the sake of simplicity let us denote them (starting from the left) by $I_1$, $I_2$ and $I_3$. Then
\begin{align*}
 I_1 \geq \varsigma e^{\sigma(\Phi(t)-\tau)}\bigl(\Phi(\tau)-\Phi(t)\bigr) 
\end{align*}
and
\begin{align*}
 I_2 \leq  \int_t^{\tau} f(u) \textup du\leq  \sup_{s \in \mathbb R} \int_s^{s+\delta} f(u) \textup du.
\end{align*}
For the integral $I_3$ we have
\begin{align*}
 I_3 & = \int^{\Phi(t)}_{t} \! \bigl(f(u)-\sigma\bigr)\! \bigl( e^{\sigma(u-t)} - e^{\sigma(u-\tau)} \bigr) \textup du \leq \bigl( e^{\sigma(\Phi(t)-t)} - e^{\sigma(\Phi(t)-\tau)} \bigr) \cdot \int_{t}^{\Phi(t)} f(u) \textup du\\[1.5mm]
 & \leq \bigl( e^{\sigma(\Phi(t)-t)} - e^{\sigma(\Phi(t)-\tau)} \bigr) k,
\end{align*}
because $0 \leq \Phi(t) - t \leq 1/\varsigma$. Since $\Phi(t)-\tau \geq t-\tau \geq -\delta$, taking into account all the above estimates, we get
\begin{align*}
  \Phi(\tau) - \Phi(t) \leq \frac{1}{\varsigma} e^{\sigma \delta} \sup_{s \in \mathbb R} \int_s^{s+\delta} f(u)\textup du + \frac{1}{\varsigma}(e^{\sigma \delta}-1)k \leq \varepsilon. 
\end{align*}
This shows that the firing map $\Phi$ is uniformly continuous. Then, of course, the displacement map $\Psi$ is also uniformly continuous.

Finally, we will show that $\inf_{t \in \mathbb R}\Psi(t)>0$. Let us observe that 
\[
1=\int_{t}^{t+\Psi(t)}(f(u)-\sigma)e^{\sigma(u-t)}\textup du\leq e^{\sigma \Psi (t)}\int_{t}^{t+\Psi(t)}(f(u)-\sigma)\textup du \leq e^{\sigma/\varsigma} \int_t^{t+\Psi(t)} f(u)\textup du,
\]
and therefore
\[
0<e^{-\sigma/\varsigma}\leq \int_{t}^{t+\Psi(t)}f(u)\textup du \qquad \text{for every $t \in \mathbb R$}.
\]
Moreover, in view of the assumption~\ref{it:mu_s1}, there is $\delta>0$ such that 
\[
 \sup_{t \in \mathbb R} \int_t^{t+\delta} f(u)\textup du \leq \frac{1}{2}e^{-\sigma/\varsigma}.
\]
If $\inf_{t\R}\Psi(t)=0$, then there would exist a real number $t_0$ such that $\Psi(t_0)\leq \delta$. So we would have
\[
0<e^{-\sigma/\varsigma}\leq \int_{t_0}^{t_0 + \Psi(t_0)} f(u) \textup du \leq \int_{t_0}^{t_0 + \delta} f(u) \textup du \leq \frac{1}{2}e^{-\sigma/\varsigma},
\]
which, clearly, is impossible. This shows that $\inf_{t\in \mathbb R}\Psi(t)>0$.
\end{proof}

Let us add that the continuity of the firing map for the PI and LIF models with locally integrable forcing term was investigated, for example, in~\cite{wmjs1} and~\cite{MS}, respectively. In particular, in~\cite{MS} the following result was established (here we added the missing assumption that the firing map $\Phi$ is well-defined)

\begin{proposition}[\cite{MS}*{Lemma~2.11 (b)}]\label{prop:ciaglosc}
Let $f \in L^1_{\text{loc}}(\mathbb R)$ and assume that the firing map $\Phi$ corresponding to the LIF model~\eqref{lif4}--\eqref{lif_spiking} is defined for every $t \in \mathbb R$. If $f(t)-\sigma>0$ for a.e. $t \in \mathbb R$, then the firing map is continuous on $\mathbb R$. 
\end{proposition}

\begin{remark}
Let us note that in Proposition~\ref{prop:ciaglosc} the assumption that $f(t)-\sigma >0$ for a.e. $t \in \mathbb R$ cannot be replaced (even if $f$ is periodic) with a weaker condition: $f(t)-\sigma>0$ for $t$ belonging to a set of positive Lebesgue measure. To see this it suffices to consider the PI model driven by the locally integrable $2$-periodic function $f\colon \mathbb R \to \mathbb R$ given by
\[
 f(t) = \begin{cases}
          1 & \text{for $t \in [2k,2k+1]$, $k \in \mathbb Z$,}\\
					0 & \text{otherwise},
				\end{cases}
\] 
since then $\Phi(0)=1$, but $\lim_{t \to 0^+}\Phi(t)=2$ (see also Example~\ref{ex:phi_f_piecewise_constant}).
\end{remark}

\subsection{Firing map and its displacement for LIF models with (almost) periodic input}
\label{sec:phi_and_psi}

In this subsection we investigate the properties of the firing map and its displacement for the LIF models driven by (almost) periodic functions. Special attention is paid to $\mu$-almost periodic inputs and the (somewhat unexpected) behaviour of such models.

We start with the following

\begin{proposition}\label{prop:mu-phi}
Let us assume that $f \in L^1_{\text{loc}}(\mathbb R)$ is $\mu$-almost periodic and such that $f(t)-\sigma \geq 0$ for a.e. $t \in \mathbb R$. The firing map $\Phi$ corresponding to the LIF model~\eqref{lif4}--\eqref{lif_spiking} is well-defined for every $t \in \mathbb R$ if and only if there exists a Lebesgue measurable set $A \subseteq \mathbb R$ with $\mu(A)>0$ such that $f(t)-\sigma>0$ for a.e. $t \in A$. 
\end{proposition}

\begin{proof}
The proof of necessity is obvious, so let us proceed to the sufficiency part. From the assumptions it follows that there exist a point $u \in \mathbb R$, together with two positive numbers $\varepsilon$, $\eta$ and a set $B \subseteq A\cap[u,u+1]$ with $\mu(B)=\varepsilon$ such that $f(t)-\sigma \geq \eta$ for a.e. $t \in B$. Reasoning analogous to that in the proof of Theorem~\ref{thm:zero} leads to the following estimate
\[
 \int_0^{2n\omega} \bigl(f(u)-\sigma\bigr)\textup du \geq \frac{\varepsilon \eta }{4}n \qquad \text{for $n \in \mathbb N$},
\]
where the number $\omega$ characterizes the relative density of the set $\mu E\{\frac{\varepsilon}{2},\frac{\eta}{2},f-\sigma\}$. Thus, we get
\[
 \limsup_{t \to +\infty}\int_0^t \bigl(f(u)-\sigma)e^{\sigma u} \textup du \geq \lim_{n \to +\infty}\int_0^{2n\omega} \bigl(f(u)-\sigma\bigr)\textup du =+\infty.
\]
This, in view of Proposition~\ref{poprawneokr eslenie}, implies that the firing map $\Phi$ is well-defined for every $t \in \mathbb R$.
\end{proof}

\begin{remark}\label{rem:TWAIN}
Let us note that Proposition~\ref{prop:mu-phi} extends Proposition~3 from~\cite{wmjs1} to LIF models and $\mu$-almost periodic forcing terms, since for a $S^1$-almost periodic function such that $f-\sigma\geq 0$ a.e. on $\mathbb R$, the existence of a set $A$ with the requested properties is equivalent with the condition: $\mathcal M\{f\}>\sigma$.

Moreover, let us also add that Proposition~\ref{prop:mu-phi} does not follow from Corollary~\ref{cor:l1-phi}, since for $\mu$-almost periodic functions the mean value may not exist (cf. Example~\ref{ex:muap_nonexistence}).
\end{remark}

Now, let us recall a well-known result concerning the periodically driven models; the idea behind this results can be traced to the paper of J. P. Keener, F. C. Hoppensteadt and J. Rinzel (see~\cite{keener1}), although no                     rigorous formulation (i.e. similar to the following one) can be found there. 

\begin{proposition}\label{factbykeener}
Let $f \in L^1_{\text{loc}}(\mathbb R)$ be a $\omega$-periodic function \textup(with $\omega>0$\textup) and let us assume that the firing map $\Phi$ corresponding to the LIF model~\eqref{lif4}--\eqref{lif_spiking} is well-defined for every $t \in \mathbb R$. Then $\Phi(t+\omega) = \Phi(t) + \omega$ for $t \in \mathbb R$, and thus the displacement map of $\Phi$ is $\omega$-periodic. 
\end{proposition}

\begin{proof}
 The proof follows easily from the properties of definite integrals of periodic functions and the fact that $\omega + \inf A= \inf(\omega +A)$ for non-empty sets $A \subseteq \mathbb R$.
\end{proof}

As observed in the above proposition, periodically driven LIF models have periodic displacement maps (which, in particular, allows to view $\Phi$ as a lift of a degree one circle map, and therefore to explore its dynamics and the properties of the spike trains by tools of circle maps theory; see e.g.~\cites{gedeon,keener1,MS,js2015}). Therefore, a natural question arises whether for an almost periodic input $f$ the corresponding firing map has always almost periodic displacement.

Before we address this question, let us, firstly, consider the case of a limit-periodic forcing term.

\begin{theorem}\label{twierdzenie o limitperiodic}
Let $f\colon \mathbb{R}\to\mathbb{R}$ be a limit-periodic function. Moreover, assume that there exists $\varsigma>0$ such that $f(t)-\sigma>\varsigma$ for $t \in \mathbb R$. Then the firing map corresponding to the LIF model~\eqref{lif4}--\eqref{lif_spiking} has limit-periodic displacement.
\end{theorem}

\begin{proof}
Since $f$ is limit-periodic, there exists a sequence $(f_n)_{n \in \mathbb N}$ of continuous periodic functions uniformly convergent to $f$ on $\mathbb R$. Clearly, we may assume that $f_n(t)-\sigma > \frac{1}{2}\varsigma$ for all $t \in \mathbb R$ and $n \in \mathbb N$.

By Proposition~\ref{poprawneokr eslenie} the firing maps $\Phi$ and $\Phi_n$ (and their displacements $\Psi$ and $\Psi_n$), which correspond to the LIF model~\eqref{lif4}--\eqref{lif_spiking} driven by $f$ and $f_n$, respectively, are well-defined for every $t \in \mathbb R$. Moreover, in view of Theorem~\ref{lem:uniform_con_bound} (or Proposition~\ref{prop:ciaglosc}) and Proposition~\ref{factbykeener}, we infer that the displacement maps $\Psi_n$ are continuous and periodic.

Let us also observe that in order to show that the displacement map $\Psi$ is limit-periodic, it suffices to show that the sequence $(\Phi_n)_{n \in \mathbb N}$ is uniformly convergent to $\Phi$ on $\mathbb R$, since $\sup_{t \in \mathbb R}\abs{\Psi(t)-\Psi_n(t)}= \sup_{t \in \mathbb R}\abs{\Phi(t)-\Phi_n(t)}$.

Given $\varepsilon>0$ choose $N \in \mathbb N$ such that $\sup_{t \in \mathbb R}\abs{f_n(t)-f(t)}<\frac{1}{2}\varsigma^2\varepsilon$ for all $n \geq N$. Let $n\geq N$ and fix $t\in\mathbb{R}$. Suppose that $\Phi_n(t)>\Phi(t)$. We calculate that
\[
 \int_{\Phi(t)}^{\Phi_n(t)} \bigl(f_n(u)-\sigma\bigr) \ue^{\sigma u}\; \textup du=\int_{t}^{\Phi(t)} \bigl(f(u)-f_n(u)\bigr) \ue^{\sigma u}\; \textup du \leq \frac{1}{2}\varsigma^2\varepsilon\ue^{\sigma \Phi(t)}(\Phi(t)-t) \leq \frac{1}{2}\varsigma\varepsilon\ue^{\sigma \Phi(t)};
\]
the last inequality in the above chain of inequalities follows from the fact that $\Phi(t)-t \leq 1/\varsigma$ for $t \in \mathbb R$ (cf. the proof of Theorem~\ref{lem:uniform_con_bound}). Simultaneously,
\[
\int_{\Phi(t)}^{\Phi_n(t)} \bigl (f_n(u)-\sigma\bigr )\ue^{\sigma u}\; \textup du \geq \frac{1}{2}\varsigma\ue^{\sigma\Phi(t)} \abs{\Phi_n(t)-\Phi(t)}.
\]
Therefore, $\abs{\Phi_n(t)-\Phi(t)}\leq \varepsilon$ for all $n \geq N$. When $\Phi(t)>\Phi_n(t)$ we arrive at the same conclusion. As a result, $\Phi$ is the uniform limit of $(\Phi_n)_{n \in \mathbb N}$. This ends the proof.
\end{proof}

Now we will pass to the case of an almost periodic input.

\begin{theorem}\label{new almost periodic displacement}
Let $f\colon \mathbb{R}\to\mathbb{R}$ be a $S^1$-almost periodic function. Moreover, assume that there exists $\varsigma>0$ such that $f(t)-\sigma>\varsigma$ for a.e. $t \in \mathbb R$. Then the firing map $\Phi$ corresponding to the LIF model~\eqref{lif4}--\eqref{lif_spiking} has uniformly almost periodic displacement $\Psi$.
\end{theorem}

\begin{remark}
Let us note that in the above result we can also assume that $f$ is $S^p$-almost periodic for some $p \in [1,+\infty)$ or even uniformly almost periodic, since $AP(\mathbb R) \subseteq S^p(\mathbb R) \subseteq S^1(\mathbb R)$ (see Remark~\ref{rem:AP_sap} and Remark~\ref{rem:s_p_s_1}).
\end{remark}

\begin{proof}[Proof of Theorem~\textup{\ref{new almost periodic displacement}}]
Clearly, the firing map $\Phi$ and the displacement map $\Psi$ are defined for every $t \in \mathbb R$. Furthermore, the maps $\Phi$ and $\Psi$ are continuous in view of Proposition~\ref{prop:ciaglosc}.

Without loss of generality we may assume that $\varsigma<1$. Given $\varepsilon >0$ let $\tau\in S^1E\{\frac{\varsigma^2\varepsilon}{4},f\}$. Then $\norm{f^{\tau}-f}_{S^1_1} <\frac{\varsigma^2  \varepsilon}{4}$. Now, fix $t \in \mathbb R$ and suppose that $\min\{\Phi(t),\Phi(t+\tau)-\tau\}=\Phi(t)$. From the definition of the firing map the following equality can be easily derived:
\begin{equation}\label{eq:prop4.9.1}
\bgabs{ \int_t^{\Phi(t)} \bigl(f(u+\tau)-f(u)\bigr)\ue^{\sigma(u+\tau)}\ud u} = \bgabs{\int_{\Phi(t)}^{\Phi(t+\tau)-\tau}\bigl(f(u+\tau)-\sigma\bigr)\ue^{\sigma(u+\tau)}\ud u}.
\end{equation}
Since $\tau\in S^1E\{\frac{\varsigma^2\varepsilon}{4}, f\}$, elementary calculations show that
\begin{equation}\label{eq:prop4.9.2}
 \bgabs{\int_{t}^{\Phi(t)} \bigl(f(u+\tau)-f(u)\bigr)\ue^{\sigma(u+\tau)}\ud u }  <\ue^{\sigma(\Phi(t)+\tau)}\frac{k \varsigma^2 \varepsilon}{4},
\end{equation}
where $k\in\mathbb{N}$ is the smallest integer such that $\Phi(t)\leq t+k$. Furthermore, since $f(u)-\sigma>\varsigma$ for a.e. $u \in \mathbb R$, we infer that
\begin{equation}\label{eq:prop4.9.3}
\int_{\Phi(t)}^{\Phi(t+\tau)-\tau} \bigl( f(u+\tau)-\sigma\bigr) \ue^{\sigma(u+\tau)}\ud u \geq \ue^{\sigma(\Phi(t)+\tau)}\varsigma \bigl(\Phi(t+\tau)-\tau-\Phi(t)\bigr).
\end{equation}
Let us note that $\Phi(u)-u\leq 1/\varsigma$ for any $u\in\mathbb{R}$, and thus $k\leq (1/\varsigma+1)$. So by~\eqref{eq:prop4.9.1}--\eqref{eq:prop4.9.3} and the above observation, we get
 \[
 \Phi(t+\tau)-\tau-\Phi(t)<\frac{k \varsigma\varepsilon}{4}\leq \frac{(\frac{1}{\varsigma}+1)\varsigma\varepsilon}{4}<\frac{\varepsilon}{2}.
 \]
The case $\min\{\Phi(t),\Phi(t+\tau)-\tau\}=\Phi(t+\tau)-\tau$ treated similarly yields $\Phi(t)-\Phi(t+\tau)+\tau<\frac{\varepsilon}{2}$. Thus  $|\Phi(t+\tau)-\Phi(t)-\tau|<\frac{\varepsilon}{2}$ for every $t \in \mathbb R$, which means that the set $E\{\varepsilon,\Psi\}$ is relatively dense as it contains the relatively dense set $S^1E\{\frac{\varsigma^2\varepsilon}{4},f\}$. This, in turn, shows that $\Psi$ is uniformly almost periodic and ends the proof.
\end{proof}

Since a uniformly almost periodic function is uniformly continuous we have the following

\begin{corollary}\label{cor:1}
Let $f\colon \mathbb{R}\to\mathbb{R}$ be a $S^1$-almost periodic function. Moreover, assume that there exists $\varsigma>0$ such that $f(t)-\sigma>\varsigma$ for a.e. $t \in \mathbb R$. Then the firing map $\Phi$ corresponding to the LIF model~\eqref{lif4}--\eqref{lif_spiking} and its displacement $\Psi$ are uniformly continuous.
\end{corollary}

\begin{remark}\label{rem:4.20}
Corollary~\ref{cor:1} is also a consequence of Theorem~\ref{lem:uniform_con_bound}, since it can be shown that every almost everywhere non-negative $S^1$-almost periodic function satisfies the condition~\ref{it:mu_s1} of Theorem~\ref{lem:uniform_con_bound} (for more details see~\cites{B-N, S}; cf. also Theorem~\ref{thm:BN_s1} below).
\end{remark}

Now, we move to the investigation of the firing map and its displacement for LIF models driven by a locally integrable $\mu$-almost periodic functions. We begin with an example showing that in the `$\mu$-almost periodic setting' it may happen that the displacement map, although uniformly continuous and bounded, fails to be almost periodic in any sense considered in this paper. This somewhat unexpected phenomenon has interesting consequences, since, for example, it allows to establish some results in the theory of almost periodicity (for more details see Section~\ref{sec:5} below).

\begin{example}\label{ex:mu_no_mu}
Let us consider the following sets: $A_n = 2^n \mathbb Z + 2^{n-1}$ for $n \in \mathbb N$ and let us observe that
\[
 \text{$A_n \cap A_m = \emptyset$ if $n \neq m$} \qquad \text{and} \qquad  \bigcup_{n \in \mathbb N} A_n = \mathbb Z \setminus\{0\}.
\]
Furthermore, for the sake of simplicity, let us put
\begin{equation}\label{eq:b_n_c_n}
 B_{n,k}(z) = \Bigl(z+\frac{k}{2^{n-1}},z + \frac{k}{2^{n-1}} + \frac{1}{2}\cdot\frac{1}{4^{n-1}}\Bigr) \qquad \text{and} \qquad C_n = \bigcup_{z \in A_n}\bigcup_{k=0}^{2^{n-1}-1} B_{n,k}(z). 
\end{equation}
Note that for a given $n \in \mathbb N$ and $z \in A_n$ we have $B_{n,k}(z) \cap B_{n,l}(z)=\emptyset$ for distinct $k,l \in \{0,\ldots,2^{n-1}-1\}$, and moreover $B_{n,k}(z) \subseteq (z,z+1)$ for $k \in \{0,\ldots,2^{n-1}-1\}$. Observe also that the sets $C_n$ are pairwise disjoint.

Now, let us define the locally integrable functions $f_n \colon \mathbb R \to \mathbb R$ by the formulas:
\[
 f_n (x)= \begin{cases}
        2^n& \text{for $x \in C_n$},\\[2mm]
        0& \text{for $x\in \mathbb R\setminus C_n$},
       \end{cases}
\]
and let us note that for each $n \in \mathbb N$ the function $f_n$ is $2^n$-periodic, since $x \in C_n$ if and only if $x \pm 2^n \in C_n$. 

We shall show that the locally integrable function $f \colon \mathbb R \to \mathbb R$ given by
\begin{equation}\label{eq:funckja}
 f(x) = 1 + \sum_{n=1}^{\infty} f_n(x) \qquad \text{for $x \in \mathbb R$},
\end{equation}
is $\mu$-almost periodic. In fact, we are going to show that $f$ is the $D$-limit of the sequence $(g_k)_{k \in \mathbb N}$ of locally integrable periodic functions, where
\[
 g_k(x)=1+\sum_{n=1}^k f_n(x) \qquad \text{for $x \in \mathbb R$}.
\]
Fix $\varepsilon>0$ and $\eta \in (0,1)$, and choose $N \in \mathbb N$ such that $2^{-(N+1)}<\varepsilon$. Then, for every $k \geq N$ and every $z \in \mathbb Z\setminus\{0\}$, we have
\[
 \set{x \in [z,z+1]}{\abs{f(x)-g_k(x)}\geq \eta} \subseteq [z,z+1]\cap C_m
\]
for some $m \geq k+1$. Hence
\[
 \mu\bigl(\set{x \in [z,z+1]}{\abs{f(x)-g_k(x)}\geq \eta} \bigr) \leq \mu\bigl([z,z+1]\cap C_m\bigr) \leq \frac{1}{2^m} \leq \frac{1}{2^{N+1}} \leq \varepsilon.
\]
Since $\set{x \in [0,1]}{\abs{f(x)-g_k(x)}\geq \eta}=\emptyset$ for every $k \in \mathbb N$, in view of Remark~\ref{lem:szacowanie_mu}, we obtain that $D(\eta;f,g_k) \leq 2\varepsilon$ for $k \geq N$. This proves that the sequence $(g_k)_{k \in \mathbb N}$ is $D$-convergent to $f$, and thus the function $f$ is $\mu$-almost periodic (see Remark~\ref{rem:D_zbieznosc2}).

Now, we would like to show that the displacement map $\Psi$ corresponding to the PI model driven by the function $f$ is uniformly continuous and bounded. To this end we are going to apply Theorem~\ref{lem:uniform_con_bound}. Since the assumption~\ref{it:mu_s2} of Theorem~\ref{lem:uniform_con_bound} is satisfied with $\varsigma=1$, it suffices to show that $f$ satisfies the assumption~\ref{it:mu_s1}. We claim that 
\begin{equation}\label{eq:szacowanie2}
 \sup_{u \in \mathbb R} \int_u^{u+\frac{1}{2^{m-1}}} f(s) \textup ds \leq \frac{1}{2^{\frac{1}{2}m-3}} \qquad \text{for $m \in \mathbb N$}.
\end{equation}
First, observe that given any $u \in \mathbb R$ and $m \in \mathbb N$, there exist $z \in \mathbb Z$ and $k \in \{0,\ldots,2^{m-1}-1\}$ such that 
\begin{equation}\label{eq:zawieranie4}
 \Bigl(u,u+\frac{1}{2^{m-1}}\Bigr) \subseteq  \Bigl(z+\frac{k}{2^{m-1}},z+\frac{k+2}{2^{m-1}}\Bigr),
\end{equation}
and therefore to obtain~\eqref{eq:szacowanie2} it suffices to estimate the integral of the function $f$ on every interval of the form  
\[
\Bigl(z+\frac{k}{2^{m-1}},z+\frac{k+1}{2^{m-1}}\Bigr), \qquad \text{where $z \in \mathbb Z$ and $k \in \{0,\ldots,2^{m-1}-1\}$}.
\]
Let us consider the following five cases.
\begin{enumerate}[wide, label={Case \arabic*:}]
 \item If $z=0$, then
 \[
   \int_{z+\frac{k}{2^{m-1}}}^{z+\frac{k+1}{2^{m-1}}} f(s) \textup ds = \frac{1}{2^{m-1}} \leq \frac{1}{2^{\frac{1}{2}m-2}}.
 \] 

	\item If $z \in A_n$ for some $n\in \mathbb N$ and $n>m$, then there exists $l \in \{0,\ldots,2^{n-1}-2^{n-m}\}$ such that
\[
 \frac{k}{2^{m-1}}=\frac{l}{2^{n-1}} < \frac{l+1}{2^{n-1}}<\ldots<\frac{l+2^{n-m}-1}{2^{n-1}} < \frac{l+2^{n-m}}{2^{n-1}}=\frac{k+1}{2^{m-1}}.
\]
Moreover, for $r \in \{0,\ldots,2^{n-m}-1\}$ we have
\[
 \Bigl(\frac{l+r}{2^{n-1}}, \frac{l+r}{2^{n-1}} + \frac{1}{2}\cdot\frac{1}{4^{n-1}} \Bigr) \subseteq  \Bigl(\frac{k}{2^{m-1}}, \frac{k+1}{2^{m-1}}\Bigr), 
\]
and the intervals
\[
   \Bigl(\frac{l+r}{2^{n-1}}, \frac{l+r}{2^{n-1}} + \frac{1}{2}\cdot\frac{1}{4^{n-1}} \Bigr), \qquad r \in \{0,\ldots,2^{n-m}-1\}, 
\]
are pairwise disjoint. So
\begin{align*}
   \int_{z+\frac{k}{2^{m-1}}}^{z+\frac{k+1}{2^{m-1}}} f(s) \textup ds & = \frac{1}{2^{m-1}} + \sum_{i=0}^{2^{n-m}-1}  \int_{B_{n,l+i}(z)} f_n(s) \textup ds = \frac{1}{2^{m-1}} + 2^{n-m}\cdot \frac{1}{2}\cdot \frac{1}{4^{n-1}}\cdot 2^n =\frac{1}{2^{m-2}} \leq \frac{1}{2^{\frac{1}{2}m-2}}.
\end{align*}

\item If $z \in A_n$ for some $n \in \mathbb N$ and $n=m$, then
\begin{align*}
   \int_{z+\frac{k}{2^{m-1}}}^{z+\frac{k+1}{2^{m-1}}} f(s) \textup ds & =\frac{1}{2^{m-1}} +  \int_{z+\frac{k}{2^{m-1}}}^{z+\frac{k+1}{2^{m-1}}} f_m(s) \textup ds =  \frac{1}{2^{m-1}} + \int_{B_{m,k}(z)} f_m(s) = \frac{1}{2^{m-2}} \leq \frac{1}{2^{\frac{1}{2}m-2}}.
\end{align*}

 \item Suppose that $z \in A_n$ for some $n\in \mathbb N$ and $n<m\leq 2n$. Then, it is easy to show that there exists exactly one number $l \in \{0,\ldots,2^{n-1}-1\}$ such that
\[
 \Bigl(z+\frac{k}{2^{m-1}}, z+\frac{k+1}{2^{m-1}}\Bigr) \subseteq \Bigl(z + \frac{l}{2^{n-1}}, z+\frac{l+1}{2^{n-1}}\Bigr). 
\]
And so
\begin{align*}
   \int_{z+\frac{k}{2^{m-1}}}^{z+\frac{k+1}{2^{m-1}}} f(s) \textup ds & \leq \frac{1}{2^{m-1}} +  \int_{z+\frac{l}{2^{n-1}}}^{z+\frac{l+1}{2^{n-1}}} f_n(s) \textup ds = \frac{1}{2^{m-1}} + \int_{B_{n,l}(z)} f_n(s) \textup ds\\[2mm]
	 & = \frac{1}{2^{m-1}} + \frac{1}{2^{n-1}} \leq \frac{1}{2^{m-1}} + \frac{1}{2^{\frac{1}{2}m-1}} \leq \frac{1}{2^{\frac{1}{2}m-2}}.
\end{align*}

\item Suppose that $z \in A_n$ for some $n\in \mathbb N$ and $2n<m$. Then 
\begin{align*}
   \int_{z+\frac{k}{2^{m-1}}}^{z+\frac{k+1}{2^{m-1}}} f(s) \textup ds & = \frac{1}{2^{m-1}} + \int_{z+\frac{k}{2^{m-1}}}^{z+\frac{k+1}{2^{m-1}}} f_n(s) \textup ds \leq \frac{1}{2^{m-1}} + \frac{1}{2^{m-n-1}} \leq \frac{1}{2^{m-1}} + \frac{1}{2^{\frac{1}{2}m-1}} \leq \frac{1}{2^{\frac{1}{2}m-2}}.
\end{align*}
\end{enumerate}

Summarizing, in each case, we have
\[
  \int_{z+\frac{k}{2^{m-1}}}^{z+\frac{k+1}{2^{m-1}}} f(s) \textup ds \leq \frac{1}{2^{\frac{1}{2}m-2}},
\]
which in connection with~\eqref{eq:zawieranie4} proves~\eqref{eq:szacowanie2}. This, in turn, implies that the function $f$ satisfies the assumption~\ref{it:mu_s1} of Theorem~\ref{lem:uniform_con_bound}, and therefore the displacement map $\Psi$ corresponding to the PI model driven by the function $f$ is bounded and uniformly continuous. 

Finally, we will show that the displacement map $\Psi$ is not $\mu$-almost periodic. First, observe that 
\[
  \int_{z}^{z+\frac{1}{2}} f(s) \textup ds \geq 1 \qquad \text{for $z \in \mathbb Z\setminus\{0\}$}.
\]
Indeed, if $z \in A_1$, then we have $B_{1, 0}(z)=\Bigl(z, z+\frac{1}{2}\Bigr)$ and
\[
\int_{z}^{z+\frac{1}{2}}f(s)\textup ds\geq \int_{z}^{z+\frac{1}{2}}f_1(s)\textup ds=1.
\]If $z \in A_n$ for some $n \geq 2$, then for $k \in \{0,\ldots,2^{n-2}-1\}$ we have
\[
 \Bigl(z + \frac{k}{2^{n-1}}, z+ \frac{k}{2^{n-1}} + \frac{1}{2}\cdot\frac{1}{4^{n-1}}\Bigr) \subseteq \Bigl[z, z +\frac{1}{2}\Bigr] 
\]
(for $k \in \{2^{n-2},\ldots,2^{n-1}-1\}$ the above inclusion does not hold), and thus
\[
 \int_z^{z+\frac{1}{2}} f(s)\textup ds = \frac{1}{2} + \int_z^{z+\frac{1}{2}}f_n(s) \textup ds = \frac{1}{2} + 2^{n-2} \cdot \frac{1}{2^{2n-1}}\cdot 2^n = 1.
\]
Moreover, let us note that $\Phi(t)=1+t$ for $t \in [-\frac{1}{4},0]$. Besides, for $t \in [z-\frac{1}{4},z]$, where $z \in \mathbb Z\setminus\{0\}$, we have $\Phi(t)-t \leq \frac{3}{4}$. In fact,
\[
 \int_t^{z+\frac{1}{2}} f(s) \textup ds \geq \int_z^{z+\frac{1}{2}}f(s) \textup ds \geq 1,
\]
and hence $\Phi(t)\leq z+\frac{1}{2}$, which shows that $\Phi(t)-t\leq \frac{3}{4}$ for every $t$ in the considered interval.

If the function $\Psi$ were $\mu$-almost periodic, then for any non-zero $\tau \in \mu E\{\frac{1}{8},\frac{1}{4},\Psi\} \cap \mathbb Z$ (the existence of such a number follows from Proposition~\ref{prop:rel_dens_e_z}), we would have
\[
 \mu \bigl(\set{x \in [-1,0]}{\abs{\Psi(x)-\Psi(x+\tau)}\geq \tfrac{1}{4}}\bigr) \leq \frac{1}{8}.
\]
However, on the other hand, since $\tau \in \mathbb Z \setminus\{0\}$, we have
\[
 [-\tfrac{1}{4},0] \subseteq \set{x \in [-1,0]}{\abs{\Psi(x)-\Psi(x+\tau)}\geq \tfrac{1}{4}},
\]
whence
\[
 \mu\bigl(\set{x \in [-1,0]}{\abs{\Psi(x)-\Psi(x+\tau)}\geq \tfrac{1}{4}}\bigr) \geq \frac{1}{4}.
\] 
The obtained contradiction shows that $\Psi$ is not $\mu$-almost periodic.
\end{example}

Our next example shows that the impression, which one might have, that every LIF (or PI) model driven by a locally integrable $\mu$-almost periodic function which is not $S^1$-almost periodic fails to have almost periodic (in some sense) displacement map is wrong. We construct a purely $\mu$-almost periodic input which gives rise to a $S^p$-almost periodic displacement map.

\begin{example}
Let us consider the following sets: $A_n = 2^n \mathbb Z + s_n$, where $s_n =\frac{1}{3}\bigl[(-2)^{n-1}-1\bigr]$ for $n \in \mathbb N$, and let us note that
\begin{equation}\label{eq:suma_an}
 \text{$A_n \cap A_m =\emptyset$ if $n\neq m$} \qquad \text{and} \qquad \bigcup_{n=1}^{\infty} A_n =\mathbb Z.
\end{equation}
Moreover, for $n \in \mathbb N$ and $z \in \mathbb Z$ put $B_n(z)=(z+1-\frac{1}{n+1},z+1)$, and for each fixed $n \in \mathbb N$ let us define the locally integrable $2^n$-periodic function $f_n \colon \mathbb R \to \mathbb R$ by the following formula
\[
 f_n(x)=\begin{cases}
          (n+1)^2, \quad & \text{if $x \in \bigcup_{z \in A_n}B_n(z)$},\\
					 0, \quad & \text{otherwise}.
				\end{cases}
\]
Similarly to Example~\ref{ex:mu_no_mu}, it can be shown that the locally integrable function $f \colon \mathbb R \to \mathbb R$, given by
\[
 f(x) = 2+\sum_{n=1}^{\infty} f_n(x) \qquad \text{for $x \in \mathbb R$},
\]
is $\mu$-almost periodic as the $D$-limit of a sequence of periodic functions. Let us note that $f$ is well-defined, since $\bigl(\bigcup_{z \in A_n}B_n(z)\bigr) \cap \bigl(\bigcup_{w \in A_m}B_m(w)\bigr)=\emptyset$ if $n\neq m$. 

It is easy to see that the function $f$ is not $S^1$-bounded, and hence it cannot be $S^1$-almost periodic (cf.~Remark~\ref{rem:sp_bounded}).

Our aim is now to show that the displacement map $\Psi$ corresponding to the PI model driven by the $\mu$-almost periodic function $f$, which clearly is well-defined for every $t \in \mathbb R$ and Lebesgue measurable, is also $\mu$-almost periodic. In fact, we are going to show that for a given $m\in \{2,3,\ldots\}$ we have
\begin{equation}\label{eq:mu_psi}
 \sup_{z \in \mathbb Z} \mu\bigl( \set{t \in [z,z+1]}{\abs{\Psi(t+2^m w)-\Psi(t)} \geq \tfrac{2}{m+1}}\bigr) \leq \frac{2}{m+1} \quad \text{for every $w \in \mathbb Z$},
\end{equation}
which would imply that $2^m \mathbb Z \subseteq \mu E\{\frac{4}{m+1}, \frac{2}{m+1}, \Psi\}$ (cf. Remark~\ref{lem:szacowanie_mu}). In particular, this would mean that for every $\varepsilon, \eta>0$ the set $\mu E\{\varepsilon, \eta, \Psi\}$ is relatively dense.

Observe also that in order to prove~\eqref{eq:mu_psi} with fixed $m  \geq 2$ and $w \in \mathbb Z$, it suffices to show that
\begin{align}\label{eq:warunek}
\begin{split}  
 & \text{for every $z \in \mathbb Z$ we have $\abs{\Psi(t+2^m w) - \Psi(t)} \leq \frac{1}{m+1}$,}\\
 &\hspace{0cm}\text{whenever $t \in [z, z+\tfrac{1}{2}-\tfrac{1}{m+1}]\cup [z+\tfrac{1}{2},z+1-\tfrac{1}{m+1})$},
\end{split}
\end{align}
since then
\begin{multline*}
 \mu\bigl(\bset{t \in [z,z+1]}{\abs{\Psi(t+2^{m} w)-\Psi(t)} \geq \tfrac{2}{m+1}} \bigr)  \leq\\[2mm]
 \leq  \mu\bigl(\bset{t \in [z,z+1]}{\abs{\Psi(t+2^{m} w)-\Psi(t)} > \tfrac{1}{m+1}} \bigr) \leq \frac{2}{m+1}.
\end{multline*}
So let us fix $m \in \{2,3,\ldots\}$ and $w \in \mathbb Z$, and take arbitrary $z \in \mathbb Z$. By~\eqref{eq:suma_an} there exists exactly one $n \in \mathbb N$ such that the integer $z$ belongs to $A_n$.
\begin{enumerate}[fullwidth, label={Case \arabic*:}]
 \item Suppose that $n \leq m$. Then, because $2^n|2^m$, we infer that $f(t+2^m w) = f(t)$ for every $t \in [z,z+\tfrac{3}{2}]$. Moreover, in view of the fact that
\begin{equation}\label{eq:phi_1}
  \int_k^{k+\frac{1}{2}} f(x)\textup dx=1 \qquad \text{for $k \in \mathbb Z$},
\end{equation}
we get $\Phi(t) \leq \Phi(z+1) \leq z +\frac{3}{2}$ for every $t \in [z,z+1]$. Hence,
\[
 \int_{t+2^m w}^{\Phi(t) + 2^m w} f(x) \textup dx = \int_t^{\Phi(t)} f(x+2^m w) \textup dx = \int_t^{\Phi(t)} f(x) \textup dx=1 \quad \text{for $t \in [z,z+1]$},
\]
which proves that $\Phi(t+2^m w) = \Phi(t)+2^m w$ for $t \in [z,z+1]$. This, in turn, shows that the condition in ~\eqref{eq:warunek} is satisfied if $z \in A_n$ with $n\leq m$.

\item Suppose now that $n >m$. First, observe that $\Phi(t)=t+\frac{1}{2}$ for $t \in [z,z+\frac{1}{2}-\frac{1}{n+1}]$, since on the interval $[z,z+1-\frac{1}{n+1}]$ the function $f$ is identically equal to $2$. Furthermore, $\Phi(t) \in [z+1-\frac{1}{n+1},z+1]$ for $t \in [z+\frac{1}{2}-\frac{1}{n+1},z+1-\frac{1}{n+1})$, which is a consequence of the following estimates:
\[
 \int_t^{z+1-\frac{1}{n+1}-\gamma} f(x) \textup dx \leq \int_{z+\frac{1}{2}-\frac{1}{n+1}}^{z+1-\frac{1}{n+1}-\gamma} f(x) \textup dx =1-2\gamma<1 \qquad \text{for every $\gamma \in (0, \tfrac{1}{2})$}
\]
and
\begin{equation}\label{eq:phi_z_1}
  \int_t^{z+1} f(x) \textup dx \geq \int_{z+1-\frac{1}{n+1}}^{z+1} f(x)\textup dx=n+1+\tfrac{2}{n+1}>1.
\end{equation} 
An easy computation shows that $z+2^m w \notin \bigcup_{i=1}^m A_i$, and thus there exists a positive integer $j >m$ such that $z+2^m w \in A_j$. Therefore, by the above reasoning, we obtain that
\[
 \Phi(t+2^m w) = t+2^m w +\frac{1}{2} \qquad \text{for $t \in [z,z+\tfrac{1}{2}-\tfrac{1}{j+1}]$}
\]
and
\[
 \Phi(t+2^m w)-2^m w \in [z+1-\tfrac{1}{j+1},z+1] \qquad \text{for $t \in [z+\tfrac{1}{2}-\tfrac{1}{j+1}, z+1-\tfrac{1}{j+1})$}.
\]
In conclusion, for $t \in [z, z+\tfrac{1}{2}-\tfrac{1}{m+1}]$ we have $\Phi(t+2^m w) - 2^m w - \Phi(t) = 0$, and moreover, for $t \in [z+\tfrac{1}{2}, z+1-\tfrac{1}{m+1})$ we have $\abs{\Phi(t+2^m w)-2^m w - \Phi(t)}\leq \tfrac{1}{m+1}$.
This proves that in the considered case the condition~\eqref{eq:warunek} is satisfied. 
\end{enumerate}

Therefore, the displacement map $\Psi$ is $\mu$-almost periodic, and since $\Psi$ is bounded (cf. the proof of Theorem~\ref{lem:uniform_con_bound}), by Remark~\ref{rem:mu_S}, we see that $\Psi$ is $S^p$-almost periodic for every $p \in [1,+\infty)$.

Finally, we are going to show that the function $\Psi$ is not uniformly continuous, and hence it cannot be uniformly almost periodic. If $z \in A_n$, then $\Phi(z+1-\frac{1}{n+1})\leq z+1$ and $\Phi(z+1)=z+\frac{3}{2}$ (cf. the formulas~~\eqref{eq:phi_1}~and~\eqref{eq:phi_z_1}), and thus
\[
 \abs{\Phi(z+1-\tfrac{1}{n+1})-\Phi(z+1)} \geq \frac{1}{2},
\]
which clearly shows that $\Psi$ is not uniformly continuous.
\end{example}

\begin{remark}
Let us add that the characterization of those locally integrable $\mu$-almost periodic inputs for which the LIF model~\eqref{lif4}--\eqref{lif_spiking} has $\mu$-almost periodic (or $S^p$-almost periodic) displacement map is still open. 
\end{remark}

We end this section with a continuity result for displacement maps corresponding to LIF models driven by $S^1$-almost periodic functions.

For simplicity, let us put
\[
 C=\set{f \in S^1(\mathbb R)}{\text{there exists $a_f>0$ such that $f(t)-\sigma>a_f$ for a.e. $t \in \mathbb R$}},
\]
and let us observe that $C$ is a convex set. Clearly, $C$ is a metric space with the metric induced by the $S^1$-norm. 

\begin{theorem}\label{oT}
The mapping $T \colon C \to AP(\mathbb R)$, which to every $S^1$-almost periodic function $f \in C$ assigns the displacement $\Psi$ of the firing map $\Phi$ corresponding to the LIF model~\eqref{lif4}--\eqref{lif_spiking} driven by $f$, is continuous.
\end{theorem}

\begin{proof}
First, let us observe that $T$ is well-defined thanks to Theorem~\ref{new almost periodic displacement}.

Given $\varepsilon \in (0,1)$ and $f \in C$ let $\delta=\varepsilon a_f \bigl(\bigl[\frac{1}{a_f}\bigr]+2\bigr)^{-1} e^{-\sigma(1+1/a_f)}$, and suppose that $\widehat{f}$ is a function in $C$ such that $\bnorm{f-\widehat{f}}_{S^1_1} \leq \delta$. If $t \in \mathbb R$ is fixed, then by the definition of the firing map, we have
\begin{equation}\label{eq:rownosc}
 \int_{t}^{\Phi(t)}\bigl(f(u)-\sigma\bigr)\ue^{\sigma u}\; \textup du=\int_{t}^{\widehat{\Phi}(t)}\bigl(\widehat{f}(u)-\sigma\bigr)\ue^{\sigma u}\;\textup du
\end{equation}
(cf. the formula~\eqref{implicitnafiringdlalif}); here $\widehat{\Phi}$ denotes the firing map corresponding to the LIF model driven by $\widehat{f}$.

Suppose that $\Phi(t)\geq \widehat{\Phi}(t)$. Then, since $\widehat{\Phi}(t)-t\leq \Phi(t)-t \leq 1/a_f$ (cf. the first part of the proof of Theorem~\ref{lem:uniform_con_bound}),  from~\eqref{eq:rownosc} it follows that
\begin{align*}
\int_{\widehat{\Phi}(t)}^{\Phi(t)} \bigl(f(u)-\sigma\bigr)\ue^{\sigma u}\; \textup du&=\int_{t}^{\widehat{\Phi}(t)} \bigl(\widehat{f}(u)-f(u)\bigr)\ue^{\sigma u}\; \textup du\\
 & \leq \ue^{\sigma \widehat{\Phi}(t)} \int_{t}^{t+[\frac{1}{a_f}]+1}\abs{\widehat{f}(u)-f(u)}\; \textup du \leq \ue^{\sigma \widehat{\Phi}(t)}\bigl(\bigl[\tfrac{1}{a_f}\bigr]+1\bigr)\delta.
\end{align*}
Simultaneously,
\begin{equation*}
\int_{\widehat{\Phi}(t)}^{\Phi(t)} \big(f(u)-\sigma\bigr)\ue^{\sigma u}\; \textup du \geq a_f\ue^{\sigma \widehat{\Phi}(t)}\bigl(\Phi(t)-\widehat{\Phi}(t)\bigr).
\end{equation*}
Thus, $\Phi(t)-\widehat{\Phi}(t) \leq \frac{1}{a_f} \bigl(\bigl[\frac{1}{a_f}\bigr]+2\bigr) \delta\leq\varepsilon$.

Now, let us assume that $\widehat{\Phi}(t)\geq \Phi(t)$. Then
\begin{align*}
 \int_{t}^{\Phi(t)+\varepsilon} \bigl(\widehat{f}(u)-\sigma\bigr) e^{\sigma u} \textup du & =  \int_{t}^{\Phi(t)+\varepsilon} \bigl(\widehat{f}(u)-f(u)\bigr) e^{\sigma u} \textup du +e^{\sigma t}+  \int_{\Phi(t)}^{\Phi(t)+\varepsilon} \bigl(f(u)-\sigma\bigr) e^{\sigma u} \textup du\\
 & \geq -  \int_{t}^{\Phi(t)+1} \abs{\widehat{f}(u)-f(u)} e^{\sigma u} \textup du + e^{\sigma t} + a_f \varepsilon e^{\sigma \Phi(t)}\\
 & \geq - e^{\sigma(\Phi(t)+1)}  \int_{t}^{\Phi(t)+1} \abs{\widehat{f}(u)-f(u)}\textup du + e^{\sigma t} + a_f \varepsilon e^{\sigma \Phi(t)}\\
 & \geq - e^{\sigma(t+1+1/a_f)}  \int_{t}^{t+[\frac{1}{a_f}]+2} \abs{\widehat{f}(u)-f(u)}\textup du + e^{\sigma t} + a_f \varepsilon e^{\sigma t}\\
 & \geq e^{\sigma t}\Bigl(1+a_f \varepsilon - \bigl(\bigl[\tfrac{1}{a_f}\bigr]+2\bigr) e^{\sigma(1+1/a_f)}\delta  \Bigr) =e^{\sigma t},
\end{align*}
which shows that $\widehat{\Phi}(t) \leq \Phi(t) + \varepsilon$.

Summing up, we have shown that for a given $t \in \mathbb R$ we have $\abs{\Phi(t)-\widehat{\Phi}(t)}\leq \varepsilon$. However, let us observe that the chosen $\delta$ does not depend on $t$, and therefore, we may conclude that $\sup_{t \in \mathbb R}\abs{\Phi(t)-\widehat{\Phi}(t)}\leq \varepsilon$. To end the proof it suffices to note that $\sup_{t \in \mathbb R}\abs{\Phi(t)-\widehat{\Phi}(t)}=\sup_{t \in \mathbb R}\abs{\Psi(t)-\widehat{\Psi}(t)}$, where $\Psi$ and $\widehat{\Psi}$ denote the displacements of the firing maps $\Phi$ and $\widehat{\Phi}$, respectively.
\end{proof}

\section{Applications to the theory of almost periodicity}
\label{sec:5}

In this very short section we give an example of an application of the results established in Section~\ref{sec:phi_and_psi} to the theory of almost periodic functions.

Let us recall the following characterization of $S^1$-almost periodic functions.

\begin{theorem}[see~\cites{B-N,S}]\label{thm:BN_s1}
Suppose that $f \in L^1_{\text{loc}}(\mathbb R)$ is a $\mu$-almost periodic function. Then $f$ is $S^1$-almost periodic if and only if
\begin{equation}\label{eq:step_mu}
 \sup_{u \in \mathbb R} \sup_{\substack{A \subseteq [u,u+1]\\[1pt] \mu(A)\leq \delta}} \int_A \abs{f(t)}\textup dt \to 0 \qquad \text{as $\delta \to 0^+$}.
\end{equation}
\end{theorem}

\begin{remark}\label{rem:s1_mu_ap}
Obviously, in~\eqref{eq:step_mu} we can replace the interval $[u,u+1]$ with any closed interval of arbitrary (but fixed) length.
\end{remark}

To the best of our knowledge, so far it has not been known whether the condition~\eqref{eq:step_mu} can be simplified, that is, whether it is possible to consider only intervals $[u,u+\delta]$ instead of any Lebesgue measurable set $A \subseteq [u,u+1]$ with $\mu(A)\leq \delta$. Our next result gives the answer to this question.

\begin{theorem}\label{thm:kontrprzyklad_warunek}
There exists a locally integrable $\mu$-almost periodic function satisfying the condition
\begin{equation}\label{eq:warunek_B_N}
 \sup_{u \in \mathbb R} \int_u^{u+\delta} \abs{f(t)} \textup dt \to 0 \qquad \text{as $\delta \to 0^+$} 
\end{equation}
which is not $S^1$-almost periodic.
\end{theorem}  

\begin{proof}
Let us consider the locally integrable $\mu$-almost periodic function $f \colon \mathbb R \to \mathbb R$ defined in Example~\ref{ex:mu_no_mu}. It was shown that $f$ satisfies the condition~\eqref{eq:warunek_B_N}. If $f$ were $S^1$-almost periodic, then due to Theorem~\ref{new almost periodic displacement}, the displacement map $\Psi$ corresponding to the PI model driven by the function $f$ would be uniformly almost periodic. However, it was shown in Example~\ref{ex:mu_no_mu} that $\Psi$ is not $\mu$-almost periodic, and hence it cannot be uniformly almost periodic. The obtained contradiction proves our claim.
\end{proof}

\section{Firing rate for the LIF model}
\label{sec:firing_rate}

The following section is devoted to the study of the firing rate for the LIF model. We start with the following

\begin{definition}
Let $t \in \mathbb R$. The limit
\[
  \FR(t) = \lim_{n \to \infty}\frac{n}{\Phi^n(t)}
\]
(whenever it exists) is called the \emph{firing rate} of the LIF model~\eqref{lif4}--\eqref{lif_spiking}; here $\Phi^n$ denotes the $n$-th iterate of the firing map corresponding to the LIF model.
\end{definition}

\begin{remark}\label{rem:zdefiniowanie}
Let us note that in the definition of the firing rate $\FR(t)$ we implicitly assume that all the iterates $\Phi^n(t)$ are well-defined. In particular, this means that the firing map $\Phi$ must be well-defined for every $t \in \mathbb R$ (cf. Remark~\ref{rem:Phi_w_d_2}).
\end{remark}

Before we proceed further, let us recall a result on the existence of the firing rate for the PI model.

\begin{theorem}[cf.~\cite{wmjs1}*{Theorem~3.2} and~\cite{brette1}*{Theorem 4}]\label{twpispecial}
Suppose that $\sigma=0$ and that $f \in L^1_{\text{loc}}(\mathbb R)$ is such that the firing map $\Phi$ corresponding to the PI model is well-defined for every $t \in \mathbb R$. If $\mathcal M \{f\}$ exists, then for every $t \in \mathbb R$ the firing rate $\FR(t)$ for the PI model also exists, and moreover $\FR(t)=\mathcal M \{f\}\in [0,+\infty]$ for $t \in \mathbb R$.
\end{theorem}

Let us consider a simple example illustrating Theorem~\ref{twpispecial}.

\begin{example}
Let $f \colon \mathbb R \to \mathbb R$ be a uniformly almost periodic function given by the formula $f(t) = 2+\cos(t)+\cos(\sqrt{2}t)$. Then the firing rate for the PI model driven by the function $f$ equals
\[
  \FR(t)=\mathcal{M}\{f\} =\lim_{T\to+\infty}\frac{1}{T}\int_{0}^T \bigl(2+\cos(u)+\cos(\sqrt{2}u)\bigr)\textup du=2 \quad \text{for $t \in \mathbb R$}.
\]
\end{example}

In some situations the existence of the firing rate implies also the existence of the mean value.

\begin{proposition}\label{prop:mean_fr}
Suppose that $\sigma=0$ and that $f \in L^1_{\text{loc}}(\mathbb R)$ is such that $f(t)\geq 0$ for a.e. $t \in \mathbb R$. If for some $s \in \mathbb R$ the firing rate $\FR(s)$ for the PI model exists\footnote{We do not exclude the case $\FR(s)=+\infty$.}, then also the mean value $\mathcal M\{f\}$ exists, and moreover $\mathcal M\{f\}=\FR(s)$.
\end{proposition}

\begin{proof}
First, let us observe that $(\Phi^n(s))_{n \in \mathbb N}$ is an increasing sequence such that $\lim_{n \to \infty} \Phi^n(s)=+\infty$ (cf. the proof of Proposition~\ref{poprawneokr eslenie}). In particular, $\FR(s) \geq 0$.

Suppose that $\FR(s)$ is finite. Given $\varepsilon >0$ let $k_0 \in \mathbb N$ be such that for all $k \geq k_0$ we have
\[
 0 < \Phi^k(s), \qquad \bgabs{\FR(s)-\frac{k}{\Phi^k(s)}}\leq \frac{1}{3}\varepsilon, \qquad \frac{k}{\Phi^k(s)} - \frac{k}{\Phi^{k+1}(s)} \leq \frac{1}{3}\varepsilon, \qquad \frac{1}{\Phi^k(s)}\leq \frac{1}{3}\varepsilon.
\] 
Then, for $T \geq \Phi^{k_0}(s)$ we have
\begin{equation}\label{eq:meanvalue_fr}
 \bgabs{\FR(s) - \frac{1}{T}\int_s^T f(u)\textup du} \leq \varepsilon.
\end{equation}
Indeed, because the sequence $(\Phi^n(s))_{n\N}$ is increasing and $\lim_{n \to \infty} \Phi^n(s)=+\infty$, for every $T\geq \Phi^{k_0}(s)$ there exists $k\geq k_0$ such that $\Phi^k(s)\leq T\leq \Phi^{k+1}(s)$, and then 
\begin{align*}
 \bgabs{ \FR(s) -\frac{1}{T}\int_{s}^T f(u) \textup du}&\leq \bgabs{\FR(s)-\frac{k}{\Phi^k(s)}} +\bgabs{\frac{k}{\Phi^k(s)}-\frac{1}{T}\int_{s}^T f(u)\textup du}\\
 & \leq \frac{1}{3}\eps+ \bgabs{\frac{k}{\Phi^k(s)}-\frac{1}{T}\int_{s}^{\Phi^k(s)}f(u) \textup du-\frac{1}{T}\int_{\Phi^k(s)}^{T}f(u)\textup du}\\
 & \leq \frac{1}{3}\eps+\bgabs{\frac{k}{\Phi^k(s)}-\frac{k}{T}}+\frac{1}{\Phi^k(s)}\int_{\Phi^k(s)}^{\Phi^{k+1}(s)}f(u)\textup du\\
 & \leq \frac{1}{3}\eps+\frac{k}{\Phi^k(s)}-\frac{k}{\Phi^{k+1}(s)}+\frac{1}{\Phi^k(s)}\leq \eps.
\end{align*}
This proves~\eqref{eq:meanvalue_fr}, and shows that $\mathcal M\{f\} = \FR(s)$. 

Now, let us assume that $\FR(s)=+\infty$. Given $N>0$, there exists $k_0\N$ such that $\Phi^{k_0}(s)> 0$ and
\[
\frac{k}{\Phi^{k+1}(s)}\geq N \quad \text{for $k\geq k_0$}.
\]
Then for $T\geq \Phi^{k_0}(s)$ we have
\[
\frac{1}{T}\int_{s}^Tf(u)\textup du\geq N.
\]
Indeed, for every $T\geq \Phi^{k_0}(s)$ there exists $k\geq k_0$ such that $\Phi^k(s)\leq T \leq \Phi^{k+1}(s)$, and then 
\[
\frac{1}{T}\int_{s}^Tf(u) \textup du\geq \frac{1}{\Phi^{k+1}(s)}\int_{s}^{\Phi^k(s)}f(u) \textup du=\frac{k}{\Phi^{k+1}(s)}\geq N.
\]
This clearly shows that $\mathcal M\{f\}=+\infty$.
\end{proof}

From Theorem~\ref{twpispecial} and Proposition~\ref{prop:mean_fr} we get the following

\begin{corollary}\label{cor:mean_fr}
Suppose that $\sigma=0$ and that $f \in L^1_{\text{loc}}(\mathbb R)$ is such that $f(t)\geq 0$ for a.e. $t \in \mathbb R$. If for some $s \in \mathbb R$ the firing rate $\FR(s)$ exists, then the mean value $\mathcal M\{f\}$ and the firing rates $\FR(t)$, $t \in \mathbb R$, exist, and moreover $\FR(t)=\mathcal M\{f\} \in [0,+\infty]$ for $t \in \mathbb R$.
\end{corollary}

Corollary~\ref{cor:mean_fr} can be viewed as a result on the invariance of the firing rate closely related with the following known 

\begin{proposition}[cf.~\cite{brette1}*{Proposition~1}]\label{prop:fr_takie_samo}
Let $f \in L^1_{\text{loc}}(\mathbb R)$ be such that $f(t)-\sigma \geq 0$ for a.e. $t\in \mathbb R$. If the firing rate $\FR(t)$ for  the LIF model~\eqref{lif4}--\eqref{lif_spiking} exists for some $t \in \mathbb R$, then it exists for every $s \in \mathbb R$ and $\FR(s)=\FR(t)$.  
\end{proposition}

\begin{remark}
Let us add that the fact that the firing rate for integrate-and-fire models with right-hand side of the form $F(x,t)$, if exists, does not depend on the initial point was observed in~\cite{brette1}. The function $F$ was assumed to be sufficiently regular so that the corresponding differential equation had a unique solution starting from any initial condition and had to be either decreasing in $x$ for all $t$, or satisfy the condition: $F(0,t)>0$ for every $t \in \mathbb R$. However, no general conditions on $F$ (apart from the very simple ones corresponding to PI models)  guaranteeing the existence of the firing rate were given.
\end{remark}

Now, we would like to find a result similar to Theorem~\ref{twpispecial} in the `LIF setting'. To this end, let us recall a few facts from the rotation theory.

\begin{definition}
The \emph{rotation number} of $t \in\mathbb{R}$ with respect to  $f\colon \mathbb{R}\to\mathbb{R}$ is defined as the limiting average displacement
\begin{displaymath}
\varrho(f,t):=\lim_{n\to\infty}\frac{f^n(t)-t}{n},
\end{displaymath}
provided the limit exists.
\end{definition}

\begin{remark}
In the theory of integrate-and-fire models the rotation number of a point $t \in \mathbb R$ with respect to the firing map $\Phi$ is called the \emph{average interspike interval}.
\end{remark}

\begin{definition}
The \emph{pointwise rotation set} of $f \colon \mathbb R \to \mathbb R$  is defined as
\[
\varrho_p(f):=\set{\varrho(f,t)}{\text{$t\in\mathbb{R}$ for which $\varrho(f,t)$ exists}}.
\]
\end{definition}

\begin{theorem}[cf.~\cite{kwapisz}*{Theorem 1 and Theorem~2}]\label{twkwapisz1}
Suppose that a non-decreasing function $f \colon \mathbb R \to \mathbb R$ admits a decomposition $f(t)=t+g(t)$, where $g \colon \mathbb R \to \mathbb R$ is uniformly is almost periodic and $\inf_{t\in \mathbb R}g(t)>0$. Then $\varrho_p(f) =\{r\}$ for some $ r\in \mathbb R$.
\end{theorem}

\begin{remark}
Let us add that J.~Kwapisz in~\cite{kwapisz} gave a thorough description of rotation sets and rotation numbers for functions with uniformly almost periodic displacements.
\end{remark}

Now, we can move to the main result of this subsection. 

\begin{theorem}\label{thm:LIF_rotation}
Let $f \colon \mathbb R \to \mathbb R$ be a $S^1$-almost periodic function and assume that there exists $\varsigma>0$ such that $f(t)-\sigma > \varsigma$ for a.e. $t \in \mathbb R$. Then for every $t \in \mathbb R$ the firing rate for the LIF model~\eqref{lif4}--\eqref{lif_spiking} driven by $f$ exists, and moreover $\FR(t)=\FR(0) \in (0,+\infty)$ for $t \in \mathbb R$.
\end{theorem}

\begin{proof}
Since $f$ is a $S^1$-almost periodic function and $f(t)-\sigma>\varsigma>0$ a.e. on $\mathbb R$, the displacement map $\Psi$ corresponding to the LIF model driven by $f$ is uniformly almost periodic (see Theorem~\ref{new almost periodic displacement}). Moreover, in view of Lemma~\ref{lem:rosnie} and Theorem~\ref{lem:uniform_con_bound} (cf. also Remark~\ref{rem:4.20}), the firing map $\Phi$ satisfies the assumptions of Theorem~\ref{twkwapisz1}, and so there is a number $r \in \mathbb R$ such that $\varrho_p(\Phi)=\{r\}$, that is, 
\[
 \lim_{n \to \infty} \frac{\Phi^n(t) -t}{n}=r \qquad \text{for some $t \in \mathbb R$}.
\]

Now, we will show that $r>0$. For every $n \in \mathbb N$ we have
\[
 \Phi^n(t) - t = \sum_{i=1}^n \bigl(\Phi^{i}(t)-\Phi^{i-1}(t)\bigr) = \sum_{i=1}^n \Psi\bigl(\Phi^{i-1}(t)\bigr) \geq n\cdot \psi,
\]
where $\psi:=\inf_{t \in \mathbb R}\Psi(t)>0$ and $\Phi^0(t):=t$, and thus $r \geq \psi>0$. This shows that the firing rate $\FR(t)$, which in this case is the multiplicative inverse of the rotation number $\varrho(\Phi,t)$ exists and is positive. To end the proof it suffices to apply Proposition~\ref{prop:fr_takie_samo}.
\end{proof}

\begin{remark}
Let us observe that if $\Psi(\mathbb R) \subseteq [a,b]$ for some positive real numbers $a,b$, then $\FR(t) \in [1/b,1/a]$, provided $\FR(t)$ exists.
\end{remark}

The question whether for LIF models the firing rate can be expressed in terms of the mean value of the drive (as it is for PI models) still remains open. The following example shows, however, that if such an expression exists it must depend not only on the mean value of the input. We are going to define two $S^1$-almost periodic functions with the same mean value which give rise to LIF models with different firing rates.

\begin{example}
Let us consider two locally integrable periodic (and thus $S^1$-almost periodic) functions $f,g \colon \mathbb R \to \mathbb R$ given by
\[
 f(t)=\begin{cases}
         2 \ & \text{for $t \in \bigl[k\ln3, k\ln3+\ln2\bigr)$, $k \in \mathbb Z$},\\
				 3 \ & \text{for $t \in \bigl[k\ln3+\ln2,(k+1)\ln3\bigr)$, $k \in \mathbb Z$} 
				\end{cases} \qquad \text{and} \qquad g(t)=3-\log_3 2.
\]
It is easy to see that $\mathcal M\{f\}=\mathcal M\{g\}=3-\log_{3}2$.

Now, we will show that the firing rates $\FR_f$ and $\FR_g$ for the LIF model~\eqref{lif4}--\eqref{lif_spiking} with $\sigma=1$ driven by $f$ and $g$, respectively, are different. First. let us note that, in view of Theorem~\ref{thm:LIF_rotation}, the firing rates $\FR_f(t)$ and $\FR_g(t)$ exist for every $t \in \mathbb R$ and $\FR_f(t)=\FR_f(0)$, $\FR_g(t)=\FR_g(0)$.

By $\Phi_f$ and $\Phi_g$ let us denote the firing maps corresponding to the inputs $f$ and $g$. Since $\Phi_f(0)=\ln 2$, $\Phi^2_f(0)=\ln3$ and $\Phi^2_f(t+\ln3)=\Phi^2_f(t)+\ln3$ for $t\R$ (cf. Proposition~\ref{factbykeener}), we infer that $\Phi^{2n}_f(0)=n \ln 3$ for $n\N$. Thus $\FR_f(0)=2/\ln 3$. Similarly,
\[
\Phi^n_g(t)=t+n\ln\biggl(1+\frac{1}{2-\log_{3}2}\biggr) \qquad \text{for $t \in \mathbb R$ and $n \in \mathbb N$},
\]
and hence
\[
\FR_g(0)=\biggl[\ln\biggl(1+\frac{1}{2-\log_{3}2}\biggr)\biggr]^{-1}.
\]

Finally, it suffices to note that $\FR_f(0)\neq \FR_g(0)$.
\end{example}

At the end of this section, let us point out that  establishing the existence and uniqueness of  the firing rate is an important issue, since it allows to qualitatively characterise the spike train of a given model via the average frequency of firing, independently of the initial condition.  This result for the LIF model in both cases of (locally integrable) periodic  and Stepanov almost periodic drive was achieved by studying the displacement map $\Psi$ and showing that it is periodic or, correspondingly, uniformly almost periodic. In the situation of a sufficiently regular periodic forcing (see~\cite{MS}), periodic displacement allowed to project the firing map onto the circle and view it as a lift of the orientation-preserving circle homeomorphism, for which the existence and uniqueness of  the rotation number follows immediately from the classical Poincar\'{e} rotation theory. However, in the almost periodic case we needed to rely on the much more recent results of \cite{kwapisz} on the maps of the real line with almost periodic displacements. Investigating the dynamics of such maps, which would lead to more detailed description of the corresponding spike trains, is also much more challenging. However, asserting the almost periodicity of the displacement map $\Psi$ opens up the possibility of using some results on different almost periodic structures (see e.g.~\cite{APTJ}).

\section{Approximation of $S^p$-almost periodic functions by Haar wavelets}
\label{sec:7}

As we have seen in Section~\ref{sec:firing_map}, giving an exact arithmetic formula for the firing map $\Phi$ is in general not an easy task. However, it is much easier to obtain the formula for $\Phi$, when the forcing term $f$ is a piece-wise constant function (see Example~\ref{ex:phi_f_piecewise_constant}). Thus, in the case of almost periodic inputs, it seems that it would be better if we could use Haar series rather than Fourier series to approximate such functions.

Our aim in this section is to provide the answer to the following question: Is it possible to approximate Stepanov almost periodic functions by Haar wavelets with any desired accuracy  in the appropriate almost periodic norm?

Before we proceed further, let us recall some facts from the wavelet theory.

\begin{definition}
The \emph{Haar wavelet} is the function $h \colon \mathbb R \to \mathbb R$ defined by the formula $h(t)=\chi_{[0,\frac{1}{2})}(t) - \chi_{[\frac{1}{2},1)}(t)$, where $\chi_A$ denotes the characteristic function of the set $A \subseteq \mathbb R$.
\end{definition}

To simplify the notation for $k \in \mathbb Z$ and $j=2^m+r$, where $m$ is a non-negative integer and $r=1,2,\ldots,2^m$, let
\[
 h_{k,1}(t) = \chi_{[k,k+1)}(t) \quad \text{and} \quad h_{k,j}(t)=2^{\frac{1}{2}m} h\bigl(2^m(t-k) - r +1\bigr).
\]
The collection $\set{h_{k,j}}{k \in \mathbb Z,\ j \in \mathbb N}$ is called the \emph{Haar system}. Since $h_{k,j} \in L^q(\mathbb R)$ for every $1 \leq q \leq +\infty$ and $\supp h_{k,j} \subseteq [k,k+1]$, given a $S^p$-almost periodic function $f \colon \mathbb R \to \mathbb R$ (where $1\leq p <+\infty$), we can define the \emph{Haar--Fourier coefficients}
\[
 a_{k,j} \zdef \int_{\mathbb R} f(u)h_{k,j}(u) \textup du.
\] 
Let us also define the projection operators $P_n$, $n \in \mathbb N$, on $S^p(\mathbb R)$ by
\begin{equation}\label{eq:HF}
 (P_n f)(t) = \sum_{k=-\infty}^{+\infty} \sum_{j=1}^n a_{k,j}h_{k,j}(t) \qquad \text{for $t \in \mathbb R$}.
\end{equation}
Observe that for each $t \in \mathbb R$ only finitely many values $h_{k,j}(t)$ are non-zero, and so the sum in~\eqref{eq:HF} is in fact finite, which means that the projections $P_nf$ are well-defined on $\mathbb R$. Moreover, it should be clear that $P_n f \in L^p_{\text{loc}}(\mathbb R)$.

For a thorough treatment of the wavelet theory we refer the Reader to~\cites{Wojtaszczyk, Walnut}.

Now, let us pass to the main part of this section. We begin with the following simple

\begin{lemma}\label{lem:7.1}
Let $[a, b), [c, d)\subseteq [0, 1)$. If $f$ is a $S^1$-almost periodic function, then the function $P \colon \mathbb R \to \mathbb R$ defined by the formula
\[
P(t)=\sum_{k=-\infty}^{+\infty}\biggl(\int_{a}^{b}f(u+k)\textup du\biggr)\chi_{[c, d)}(t-k), \quad t \in \mathbb R,
\]
is $S^p$-almost periodic for every $p \in [1,+\infty)$.
\end{lemma}

\begin{proof}
First, let us note that the function $P$ is well-defined, since for every $t \in \mathbb R$ the value of $\chi_{[c, d)}(t-k)$ is non-zero for at most one $k \in \mathbb Z$. Moreover, $P \in L^p_{\text{loc}}(\mathbb R)$ for every $p \in [1,+\infty)$.

Let us fix an arbitrary $\varepsilon>0$ and let $\tau \in S^1E\{\frac{1}{2}\eps, f\}\cap \mathbb Z$. Given $t \in \mathbb R$, if $t-[t] \in [0, 1)\setminus [c, d)$, then $P(t+\tau)=P(t)$. On the other hand, if $t-[t] \in [c,d)$, then 
\[
\abs{P(t+\tau)-P(t)}=\bgabs{\int_{a}^{b}f(u+[t]+\tau) \textup du-\int_{a}^{b}f(u+[t])\textup du}\leq \int_{[t]}^{[t]+1}\babs{f(u+\tau)-f(u)} \textup du\leq \frac{1}{2}\eps.
\]
Therefore,
\[
\sup_{s\R} \biggl(\int_{s}^{s+1}\babs{P(u+\tau)-P(u)}^p \textup du\biggr)^{1/p} < \eps,
\]
which means that $S^1E\{\frac{1}{2}\eps, f\}\cap \mathbb Z \subseteq S^pE\{\eps, P\}$. Thus, $P$ is $S^p$-almost periodic (cf. Proposition~\ref{prop:rel_dens_e_z} and Remark~\ref{rem:rel_dens_e_z}).
\end{proof}

\begin{remark}
Let us note that from the proof of Lemma~\ref{lem:7.1} it follows that for every $\eps>0$ the set $E\{\eps,P\}$ is relatively dense.
However, in general the function $P$ is  not uniformly almost periodic, since it may happen that $P$ is not continuous.
\end{remark}

\begin{corollary}\label{lem:7.2}
If $f$ is $S^1$-almost periodic, then  for every $n \in \mathbb N$ the projection $P_nf$ is $S^p$-almost periodic with any $p \in [1,+\infty)$.
\end{corollary}

\begin{proof}
From Lemma~\ref{lem:7.1} we know that the function $h_1 \colon \mathbb R \to \mathbb R$ defined as
\[
h_1(t) := \sum_{k=-\infty}^{+\infty}a_{k, 1}h_{k, 1}(t)=\sum_{k=-\infty}^{+\infty}\biggl(\int_{0}^{1}f(u+k)\textup du\biggr)\chi_{[0, 1)}(t-k),
\]
is $S^p$-almost periodic for any $p \in [1, +\infty)$. 

Now, let $j\geq 2$ be of the form $j=2^m+r$, where $m$ is a non-negative integer and $r \in \{1,\ldots,2^m\}$. Since for every $k \in \mathbb Z$ and $t \in \mathbb R$ we have
\begin{align*}
 a_{k,j}h_{k,j}(t) & = 2^{m}\biggl(\int_{\frac{r-1}{2^m}}^{\frac{2r-1}{2^{m+1}}}f(u+k) \textup du\biggr)\chi_{[\frac{r-1}{2^m}, \frac{2r-1}{2^{m+1}})}(t-k)\\
&\qquad -2^{m}\biggl(\int_{\frac{r-1}{2^m}}^{\frac{2r-1}{2^{m+1}}}f(u+k)\textup du\biggr)\chi_{[\frac{2r-1}{2^{m+1}}, \frac{r}{2^m})}(t-k)\\
& \qquad -2^{m}\biggl(\int_{\frac{2r-1}{2^{m+1}}}^{\frac{r}{2^m}}f(u+k)\textup du\biggr)\chi_{[\frac{r-1}{2^m}, \frac{2r-1}{2^{m+1}})}(t-k)\\
&\qquad +2^{m}\biggl(\int_{\frac{2r-1}{2^{m+1}}}^{\frac{r}{2^m}}f(u+k)\textup du\biggr)\chi_{[\frac{2r-1}{2^{m+1}}, \frac{r}{2^m})}(t-k),
\end{align*}
the function $h_j \colon \mathbb R \to \mathbb R$ defined as
\[
h_j(t):= \sum_{k=-\infty}^{+\infty}a_{k, j}h_{k, j}(t),
\]
is $S^p$-almost periodic with any $p \in [1,+\infty)$.

To end the proof it suffices to note that $(P_n f)(t)=\sum_{j=1}^{n}h_j(t)$ for $t \in \mathbb R$.
\end{proof}

Now, we are in position to prove the main theorem of this section.

\begin{theorem}
Let $p \in [1,+\infty)$. Given a $S^p$-almost periodic function $f \colon \mathbb R \to \mathbb R$, the sequence of projections $(P_n f)_{n \in \mathbb N}$ converges to $f$ with respect to the $\norm{\cdot}_{S^p_1}$-norm.
\end{theorem}

\begin{proof}
First, let us note that the norm $\norm{\cdot}_{S^p_1}$ is equivalent to the norm $\absp{\cdot}_{S^p_1}$ given by
\[
 \absp{f}_{S^p_1} = \sup_{l \in \mathbb Z} \biggl( \int_l^{l+1} \abs{f(u)}^p \textup du \biggr)^{1/p}
\]
(cf. Remark~\ref{lem:szacowanie_mu}). Therefore, for a given $l \in \mathbb Z$, in view of~\cite{uljanov}*{Theorem~11} (see also~\cite{C}*{Theorem~7}), we have
\begin{align*}
 \biggl(\int_l^{l+1} \abs{(P_n f)(u) - f(u)}^p \textup du\biggr)^{1/p} & = \biggl(\int_l^{l+1} \bgabs{\sum_{j=1}^n a_{l,j}h_{l,j}(u) - f(u)}^p \textup du\biggr)^{1/p}\\
& \leq 24 \sup_{0\leq h\leq \frac{1}{n}} \biggl(\int_l^{l+1-h} \abs{f(u+h)-f(u)}^p \biggr)^{1/p}\\
 & \leq 24 \sup_{0 \leq h \leq \frac{1}{n}} \sup_{l \in \mathbb Z}  \biggl(\int_l^{l+1} \abs{f(u+h)-f(u)}^p \biggr)^{1/p}\\
& = 24\sup_{0 \leq h \leq \frac{1}{n}} \absp{f-f^h}_{S^p_1}.
\end{align*}
But $S^p$-almost periodic functions are $S^p$-continuous, which means that $\absp{f-f^h}_{S^p_1} \to 0$ as $h \to 0$ (see for example~\cite{Stoinski}*{Theorem~2.1} or~\cite{Levitan}*{Theorem~5.2.3}). Thus
\[
 \absp{P_n f-f}_{S^p_1} \leq 24\sup_{0 \leq h \leq \frac{1}{n}} \absp{f-f^h}_{S^p_1} \to 0 \qquad \text{as $n \to +\infty$}.
\]
This shows our claim and ends the proof.
\end{proof}

\begin{remark}
Let us point out that in our approach we have used the same basic Haar system for all the intervals $[k,k+1)$, although the Haar--Fourier coefficients had to be calculated separately for each interval. 

For completeness, let us also add that  the problem of wavelet approximations of certain almost periodic functions was investigated, for example, in \cites{falki2,kim,falki1}.
\end{remark}

\appendix
\section{Some remarks on the mean value}
\label{app:mean_value}

In this appendix we would like to address very briefly a problem which is not directly connected with the study of integrate-and-fire models, but which seems interesting from the point of view of the general theory of ordinary differential equations in spaces of almost periodic functions. 

It is known that if $f \in AP(\mathbb R)$, then its antiderivative, that is the function $F(t) \zdef \int_0^t f(s)\textup ds$, is uniformly almost periodic if and only if it is bounded (see, for example,~\cite{corduneanu}*{Theorem~4.1}). For this reason there are plenty of examples of uniformly almost periodic functions whose antiderivatives fail to be uniformly almost periodic (take for example the function $f(t)=\frac{1}{2}+\sin{t}$). As a consequence, it is not so straightforward (in general) to apply the fixed point approach in order to establish existence (and uniqueness) results for differential and integral equations in classes of almost periodic functions. 

Given a uniformly almost periodic function $f \colon \mathbb R \to \mathbb R$ we would like to find a real number $m$ such that the function $g \colon \mathbb R \to \mathbb R$ defined by $g(t)=f(t)-m$ admits a uniformly almost periodic antiderivative. Let us observe that from the fact that a uniformly almost periodic function with bounded antiderivative must have zero mean value, it follows that the sole candidate for the number $m$ is the mean value $\mathcal M\{f\}$. 

First, let us consider the case when $f$ is a locally integrable periodic function, although the class of such functions is not contained in $AP(\mathbb R$).

\begin{proposition}\label{prop:m}
Let $f \colon \mathbb R \to \mathbb R$ be a locally integrable periodic function with a period $\omega >0$. Then the antiderivative $F$ of the function $t \mapsto f(t)-\mathcal M\{f\}$ is bounded. 
\end{proposition}

\begin{proof}
Let us recall that in the case of locally integrable periodic functions, the mean value can be expressed in a simpler form, namely
\[
 \mathcal M\{f\}=\frac{1}{\omega}\int_0^\omega f(s)\textup ds
\]
(see~\cite{Zaidman}*{Remark, p.~88}). Fix $t \geq 0$. Then $t=k\omega+h$ for some $k \in \mathbb N\cup\{0\}$ and $h \in [0,\omega)$. Since
\[
 \int_0^t f(s)\textup ds = \sum_{n=1}^k \int_{(n-1)\omega}^{n\omega} f(s) \textup ds + \int_{k\omega}^t f(s) \textup ds = k\int_0^\omega f(s) \textup ds + \int_0^h f(s) \textup ds,
\]
we get
\begin{align*}
 \bgabs{\int_0^t f(s) \textup ds - \mathcal M\{f\}t} & = \bgabs{k\int_0^\omega f(s) \textup ds+ \int_0^h f(s) \textup ds - \frac{t}{\omega} \int_0^\omega f(s)\textup ds }\\
 & = \bgabs{\int_0^h f(s) \textup ds - \frac{h}{\omega} \int_0^\omega f(s) \textup ds} \leq 2\int_0^\omega \abs{f(s)} \textup ds.
\end{align*} 
The proof for $t <0$ is analogous.
\end{proof}

From Proposition~\ref{prop:m} we immediately get the following

\begin{corollary}
If $f \colon \mathbb R \to \mathbb R$ is a generalized trigonometric polynomial, that is, $f(t) = \sum_{j=1}^n\bigl( a_j \sin(\lambda_j t) + b_j\cos(\lambda_j t)\bigr)$, where $a_j,b_j,\lambda_j \in \mathbb R$, then the antiderivative of the function $t \mapsto f(t)-\mathcal M\{f\}$ is bounded.
\end{corollary}

However, in general, the antiderivative of the function $t \mapsto f(t) - \mathcal M \{f\}$, where $f \in AP(\mathbb R)$, may be unbounded as is shown by the following

\begin{example}[cf.~\cite{Stoinski}*{pp.~39-40} or~\cite{corduneanu}*{Corollary~2, p.~31}]
Let us consider the function $f \colon \mathbb R \to \mathbb R$ defined by the following formula
\[
 f(t) =  -\sum_{n=1}^{\infty} \frac{1}{n^2} \sin\biggl(\frac{t}{n^2}\biggr) \qquad \text{for $t \in \mathbb R$},
\]
Since the above series is uniformly convergent on $\mathbb R$, the function $f$ is uniformly almost periodic. Moreover, $\mathcal M\{f\}=0$. 

However, the antiderivatife $F$ of $f$ is not uniformly almost periodic, since if it were, then its Fourier series would be of the form
\[
 c + \sum_{n=1}^{\infty} \cos\biggl(\frac{t}{n^2}\biggr)
\]
for some $c\in \mathbb R$, which is impossible due to the fact that the sequence of coefficients of the Fourier series of a uniformly almost periodic function is convergent to zero (see~\cite{Stoinski}*{p.~37} or~cf.~\cite{corduneanu}*{Theorem~1.18}). 
\end{example}

\begin{remark}
It turns out that the problem addressed in this appendix is closely connected with the study of the class of the so-called $(IC)$-almost periodic functions, since the function $t \mapsto f(t) - \mathcal M \{f\}$, where $f \colon \mathbb R \to \mathbb R$ is uniformly almost periodic, admits bounded antiderivative if and only if $f - \mathcal M \{f\}$ is $(IC)$-almost periodic.

We refer the Reader interested in $(IC)$-almost periodic functions to the paper~\cite{A1} or to the monograph~\cite{Stoinski}*{Section~3.2}. The discussion concerning the relationship between almost periodic functions and their antiderivatives can be also found in~\cites{St3, St4}.
\end{remark}

\subsection*{Acknowledgements}
We are indebted to D.~Bugajewski from Adam Mickiewicz University for encouraging A.~Nawrocki and P.~Kasprzak to investigate integrate-and-fire models and their applications and for stimulating discussions concerning the paper. We are also indebted to A.~Kamont from Institute of Mathematics of the Polish Academy of Sciences, who has shown J. Signerska-Rynkowska the method of approximating Stepanov almost periodic functions by Haar wavelets.  Last but not least, we thank P.~Ma\'ckowiak from Pozna\'n University of Economics for his question concerning the relationship between the mean value and the antiderivative of an almost periodic function, which led to the formulation of the results stated in Appendix.

The last author gratefully acknowledge the support received from National Science Centre within the grant 2014/15/B/ST1/01710.

\end{document}